\documentclass{amsart}

\usepackage{a4wide,bbding, enumerate, amsmath, amsfonts, amssymb, amsthm, wasysym, stmaryrd, graphics, graphicx, xcolor, url, hyperref, hypcap, xargs, overpic}
\hypersetup{colorlinks=true, citecolor=darkblue, linkcolor=darkblue}
\graphicspath{{figures/}}

\title{Associahedra via spines}

\thanks{
Vincent Pilaud was partially supported by grant MTM2011-22792 of the Spanish MICINN, by the French ANR grant EGOS (12 JS02 002 01), and by European Research Project ExploreMaps (ERC StG 208471). \\ 
\indent Carsten Lange was supported by the TEOMATRO grant ANR-10-BLAN 0207.}

\author{Carsten Lange}
\address{Technische Universit\"at M\"unchen}
\email{clange@math.fu-berlin.de}
\urladdr{http://geom.mi.fu-berlin.de/lange/}

\author{Vincent Pilaud}
\address{CNRS \& \'Ecole Polytechnique}
\email{vincent.pilaud@lix.polytechnique.fr}
\urladdr{http://www.lix.polytechnique.fr/~pilaud/}

\newtheorem{theorem}{Theorem}
\newtheorem{corollary}[theorem]{Corollary}
\newtheorem{proposition}[theorem]{Proposition}
\newtheorem{lemma}[theorem]{Lemma}

\theoremstyle{definition}
\newtheorem{example}[theorem]{Example}
\newtheorem{remark}[theorem]{Remark}

\newtheorem{example*}[]{Example}
\newenvironment{exm}[2]{\setcounter{example*}{#1-1}\begin{example*}[#2]}{\end{example*}}

\newcommand{\R}{\mathbb{R}} 
\newcommand{\fS}{\mathfrak{S}} 
\newcommand{\cN}{\mathcal{N}} 

\newcommand{\set}[2]{\left\{ #1 \;\middle|\; #2 \right\}} 
\newcommand{\bigset}[2]{\big\{ #1 \;\big|\; #2 \big\}} 
\newcommand{\biggset}[2]{\bigg\{ #1 \;\bigg|\; #2 \bigg\}} 
\newcommand{\ssm}{\smallsetminus} 
\newcommand{\dotprod}[2]{\langle \, #1 \, | \, #2 \, \rangle} 
\newcommand{\symdif}{\, \triangle \,} 
\newcommand{\eqdef}{\mbox{\,\raisebox{0.2ex}{\scriptsize\ensuremath{\mathrm:}}\ensuremath{=}\,}} 
\newcommand{\defeq}{\mbox{~\ensuremath{=}\raisebox{0.2ex}{\scriptsize\ensuremath{\mathrm:}} }} 
\newcommand{\polar}{^\diamond} 
\newcommand{\simplex}{\triangle} 

\newcommandx{\Asso}[1][1=\ver{P}]{\mathsf{Asso}(#1)} 
\newcommandx{\Perm}[1][1=n]{\mathsf{Perm}(#1)} 
\newcommandx{\Para}[1][1=n]{\mathsf{Para}(#1)} 
\newcommandx{\Defo}[1][1={\{z_J\}_{\varnothing \ne J \subseteq [n+1]}}]{\mathsf{Defo}\!\left(#1\right)} 
\newcommand{\chor}{\mathsf{R}} 
\newcommand{\Abo}{\mathsf{A}} 
\newcommand{\Bel}{\mathsf{B}} 
\newcommand{\abo}{\mathsf{a}} 
\newcommand{\bel}{\mathsf{b}} 
\newcommand{\Down}{\mathsf{D}} 
\newcommand{\Up}{\mathsf{U}} 
\newcommand{\up}{\mathsf{u}} 
\newcommand{\ver}[1]{\mathbf{#1}} 
\newcommand{\nod}[1]{\mathbf{#1}^*} 
\newcommand{\spine}[1]{{#1}^*} 
\newcommand{\leveledSpine}[1]{{#1}^\circledast} 
\newcommand{\Out}{\mathrm{out}} 
\newcommand{\In}{\mathrm{in}} 
\newcommand{\conv}{\mathsf{conv}} 
\newcommand{\minTriang}{{\underline T}} 
\newcommand{\minDiag}{{\underline \delta}} 
\newcommand{\maxTriang}{{\overline T}} 
\newcommand{\maxDiag}{{\overline \delta}} 
\newcommand{\flipGraph}{\mathcal{G}} 
\newcommand{\surjectionPermAsso}{\kappa} 
\newcommand{\surjectionAssoPara}{\lambda} 
\newcommand{\leveledBijection}[1]{\kappa(#1)^\circledast} 
\newcommand{\face}{\b{f}} 
\newcommand{\Euler}{\mathrm{Eul}} 
\newcommand{\Narayana}{\mathrm{Nar}} 
\newcommand{\contact}[1]{{#1}^{\#}} 
\newcommand{\brick}{\mathsf{b}} 
\newcommand{\Brick}{\mathsf{B}} 
\newcommand{\dihedral}{\mathbb{D}} 
\newcommand{\lef}{\textsf{l}} 
\newcommand{\mi}{\textsf{m}} 
\newcommand{\rig}{\textsf{r}} 

\DeclareMathOperator{\cone}{cone} 

\newcommand{\fref}[1]{Figure~\ref{#1}} 
\newcommand{\ie}{\textit{i.e.}~} 
\newcommand{\eg}{\textit{e.g.}~} 
\newcommand{\ordinal}{\textsuperscript{th}} 
\newcommand{\ex}{^{\textrm{ex}}} 
\newcommand{\loday}{^{\textrm{Lod}}} 
\definecolor{darkblue}{rgb}{0,0,0.7} 
\renewcommand{\b}[1]{\mathbf{#1}} 
\newcommand{\darkblue}{\color{darkblue}} 
\newcommand{\defn}[1]{\emph{\darkblue #1}} 
\usepackage{todonotes}

\renewcommand{\paragraph}[1]{\bigskip\noindent\textbf{#1 ---}}

\definecolor{darkgreenpic}{rgb}{0,0.54,0.18} 
\definecolor{darkbluepic}{rgb}{0.04,0.18,0.54} 

\begin{document}

\vspace*{-1.1cm}

\begin{abstract}
An associahedron is a polytope whose vertices correspond to triangulations of a convex polygon and whose edges correspond to flips between them. Using labeled polygons, C.~Hohlweg and C.~Lange constructed various realizations of the associahedron with relevant properties related to the symmetric group and the permutahedron. We introduce the spine of a triangulation as its dual tree together with a labeling and an orientation. This notion extends the classical understanding of the associahedron via binary trees, introduces a new perspective on C.~Hohlweg and C.~Lange's construction closer to J.-L.~Loday's original approach, and sheds light upon the combinatorial and geometric properties of the resulting realizations of the associahedron. It also leads to noteworthy proofs which shorten and simplify previous approaches.

\medskip
\noindent
{\sc keywords.}
Associahedra, polytopal realizations, generalized permutahedra.
\end{abstract}

\maketitle

Associahedra were originally defined as combinatorial objects by J.~Stasheff in~\cite{Stasheff}, and later realized as convex polytopes through several different geometric constructions~\cite{Lee, GelfandKapranovZelevinsky, BilleraFillimanSturmfels, Loday, HohlwegLange, PilaudSantos-brickPolytope, CeballosSantosZiegler}, reflecting deep connections to a great variety of mathematical disciplines. They belong to the world of Catalan combinatorics, which can be described by several different equivalent models, all counted by the Catalan number $C_{n+1} = \frac{1}{n+2}\binom{2n+2}{n+1}$, such as parenthesizings of a non-associative product, Dyck paths, binary trees, triangulations, non-crossing partitions, etc~\cite[Exercice~6.19]{Stanley}. Among the plethora of different approaches, the following description suits best our purposes. An \defn{$n$-dimensional associahedron} is a simple polytope such that the inclusion poset of its non-empty faces is isomorphic to the reverse inclusion poset of the dissections of a convex~$(n+3)$-gon~$\ver{P}$ (\ie the crossing-free sets of internal diagonals of~$\ver{P}$). See \fref{fig:associahedra} for $3$-dimensional examples. For our purposes, it is important to keep in mind that the vertices of the associahedron correspond to the triangulations of~$\ver{P}$, its edges correspond to the flips between these triangulations, and its facets correspond to the internal diagonals of~$\ver{P}$. Moreover a vertex belongs to a facet if and only if the corresponding triangulation contains the corresponding internal diagonal.

\begin{figure}[h]
  \centerline{\includegraphics[width=.9\textwidth]{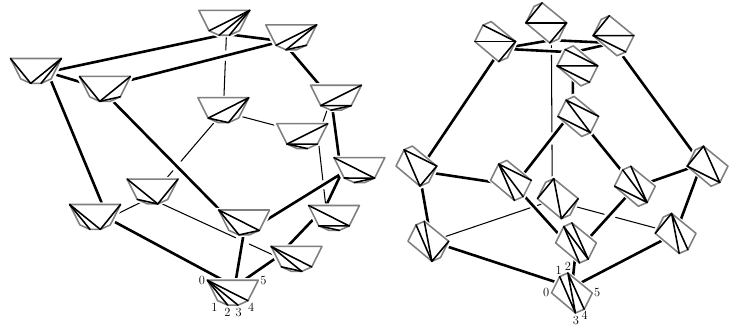}}
  \caption{Two polytopal realizations of the $3$-dimensional associahedron with vertices labeled by triangulations of convex hexagons described in~\cite{HohlwegLange}.}
  \label{fig:associahedra}
\end{figure}

S.~Shnider and S.~Sternberg~\cite{ShniderSternberg} (see also~\cite[Appendix~B]{StasheffShnider}) found an elegant construction of associahedra, and J.-L.~Loday~\cite{Loday} gave an explicit combinatorial formula for their vertex coordinates in terms of binary trees. Generalizing this construction, C.~Hohlweg and C.~Lange~\cite{HohlwegLange} found various realizations of associahedra that are parametrized by certain labelings of the convex polygon~$\ver{P}$. Their associahedra have various remarkable combinatorial properties. First, they have particularly simple vertex and facet descriptions: their integer vertex coordinates can be expressed by simple formulas on the corresponding triangulations of the polygon, and their normal vectors are given by the characteristic vectors of sets naturally associated to the internal diagonals of the polygon. Second, they are deeply connected to the symmetric group, the braid arrangement, and the permutahedron. To illustrate these connections, let us just observe here that
\begin{enumerate}[(i)]
\item these associahedra are obtained from the permutahedron by gliding some facets to infinity, 
\item the normal fans of these associahedra coarsen that of the permutahedron \mbox{(braid arrangement)},
\item the vertex barycenters of these associahedra all coincide with that of the permutahedron.
\end{enumerate}
More details on these connections are discussed below, see also~\cite{Hohlweg} for a recent survey.

In this paper, we focus on~\defn{spines} of triangulations, \ie on their oriented and labeled dual trees. These spines allow a new perspective on the realizations constructed by C.~Hohlweg and C.~Lange and refine the classical description of associahedra in terms of binary trees. We believe that this approach has several advantages. First, it extends the original presentation of J.-L.~Loday~\cite{Loday} and avoids the detour via triangulations needed in~\cite{HohlwegLange}. Second, spines naturally encode the normal fan of these associahedra and therefore shed light upon the above-mentioned connection between the permutahedron and the associahedra. Third, this approach leads to short and unified proofs of certain results on these associahedra: in particular, we obtain noteworthy proofs that these polytopes indeed realize the associahedron (Theorem~\ref{theo:associahedron}), of simple formulas for the dilation factors in certain Minkowski decompositions of these associahedra (Theorem~\ref{theo:MinkowskiDilationFactors}), and that their barycenter coincide with that of the permutahedron (Theorem~\ref{theo:barycenter}). Finally, our interpretation in terms of spines opens the door to further extensions of this construction to signed tree associahedra~\cite{Pilaud}. 

The paper is organized in six sections. In Section~\ref{sec:associahedron}, we describe the polytopes of~\cite{HohlwegLange} in terms of spines and provide a complete, independent and short proof that they indeed realize the associahedron. The next four sections explore further properties of these associahedra from the perspective of spines. Throughout these sections, we only include detailed proofs for which spines provide new insights, in order to keep the paper short and focussed on spines. In Section~\ref{sec:flips}, we study dissections, refinements and flips, and discuss some properties of the slope increasing flip lattice. Section~\ref{sec:normalFan} explores the connection between these associahedra and the permutahedron, in particular the relations between their normal fans. In Section~\ref{sec:Minkowski}, we decompose these associahedra as Minkowski sums and differences of dilated faces of the standard simplex. We discuss further topics in Section~\ref{sec:invariants}, where we present in particular a concise proof that the vertex barycenters of the associahedra and permutahedron coincide. Finally, in Section~\ref{sec:existingWork}, we relate spines to various topics from the literature, including detailed references to existing works, and we discuss further possible applications of the notion of spines.

\vspace{.3cm}
\textit{This paper is dedicated to two fiftieth birthdays. First, to that of G\"unter M. Ziegler to acknowledge his deep influence on the theory of polytopes. Second, to the discovery of the associahedron by J.~Stasheff~\cite{Stasheff} which still motivates nowadays an active research area.}


\section{Spines and associahedra}
\label{sec:associahedron}


\subsection{Triangulations and their spines}

Let~$\ver{P} \subset \R^2$ be a convex $(n+3)$-gon with no two vertices on the same vertical line. We denote by~$\ver{0}, \ver{1}, \dots, \ver{n+2}$ the vertices of~$\ver{P}$ ordered by increasing \mbox{$x$-coordi}\-nate. We call \defn{up vertices} and \defn{down vertices} the intermediate vertices~$\ver{1}, \dots, \ver{n+1}$ on the upper and lower hull of~$\ver{P}$ respectively. Similarly, we call \defn{up labels}~$U \subseteq [n+1]$ and \defn{down labels}~$D \subseteq [n+1]$ the label sets of all up and down vertices respectively. Note that the leftmost and rightmost vertices~$\ver{0}$ and~$\ver{n+2}$ are neither up nor down vertices and, consequently,~$0$ and $n+2$ are neither up nor down labels.

Let~$\delta$ be a diagonal of~$\ver{P}$ (internal or not). We denote by~$\Abo(\delta)$ the set of labels~${j \in [n+1]}$ of the vertices~$\ver{j}$ which lie above the line supporting~$\delta$ where we include the endpoints of~$\delta$ if they are down, and exclude them if they are up. Similarly, we denote by~$\Bel(\delta)$ the set of labels~${j \in [n+1]}$ of the vertices~$\ver{j}$ which lie below the line supporting~$\delta$ where we include the endpoints of~$\delta$ if they are up, and exclude them if they are down. Note that $\Abo(\delta)$ and~$\Bel(\delta)$ partition~$[n+1]$, and that we never include the labels~$0$ and~$n+2$ in~$\Abo(\delta)$ or~$\Bel(\delta)$. We set~$\abo(\delta) \eqdef |\Abo(\delta)|$ and $\bel(\delta) \eqdef |\Bel(\delta)|$. See Example~\ref{exm:small1} and~\fref{fig:flip}.

Let~$T$ be a triangulation of~$\ver{P}$. The \defn{spine} of~$T$ is its oriented and labeled dual tree~$\spine{T}$,~with 
\begin{itemize}
\item an internal node~$\nod{j}$ for each triangle~$\ver{i} \ver{j} \ver{k}$ of~$T$ where~${i<j<k}$,
\item an arc~$\delta^*$ for each internal diagonal~$\delta$ of~$T$, oriented from the triangle below~$\delta$ to the triangle above~$\delta$, and
\item an outgoing (resp.~incoming) blossom~$\delta^*$ for each top (resp.~bottom) boundary edge~$\delta$ of~$\ver{P}$.
\end{itemize}
The \defn{up nodes} and \defn{down nodes} are the internal nodes of~$\spine{T}$ labeled by up and down labels respectively. Observe that an up node has indegree one and outdegree two, while a down node has indegree two and outdegree one. See Example~\ref{exm:small1} and~\fref{fig:flip}.

\begin{example}
\label{exm:small1}
To illustrate these definitions, consider the decagon~$\ver{P}\ex$ with~$U\ex = \{2,3,5,7\}$ and $D\ex = \{1,4,6,8\}$ represented in \fref{fig:flip}. We have \eg~$\Abo(\ver{2}\ver{7}) = \{3,5\}$ and~$\Bel(\ver{2}\ver{8}) = \{1,2,4,6\}$. Two triangulations~$T\ex$ (left) and~$\widetilde T\ex$ (right) of~$\ver{P}\ex$ and their spines are represented. Throughout the paper, the vertices of the polygon are represented with dots~\scalebox{1.2}{$\bullet$}, while the up and down nodes of the spines are respectively represented with up and down triangles~$\blacktriangle$ and~$\blacktriangledown$.

\begin{figure}
  \centerline{\includegraphics[width=\textwidth]{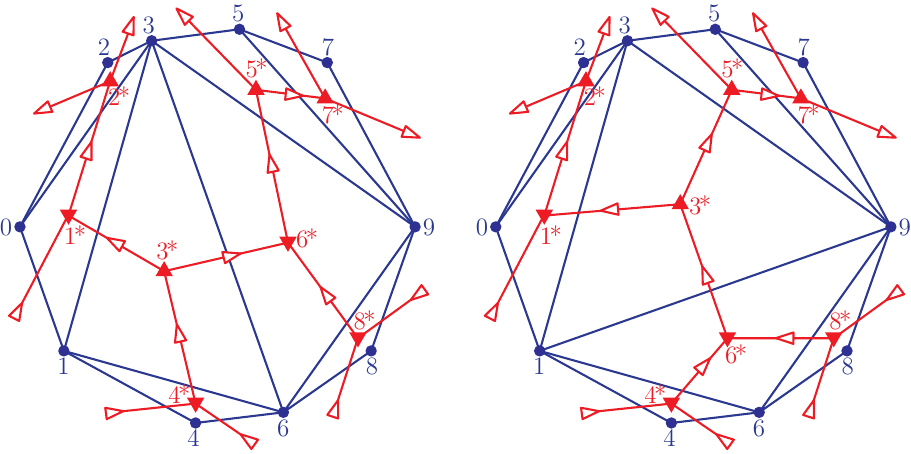}}
  \caption{Two triangulations~$T\ex$ and~$\widetilde T\ex$ of the polygon~$\ver{P}\ex$ and their spines.}
  \label{fig:flip}
\end{figure}
\end{example}

\begin{example}[Loday's associahedron]
\label{exm:Loday1}
Let~$\ver{P}\loday$ denote the $(n+3)$-gon for which~${U\loday = \varnothing}$ and ${D\loday = [n+1]}$. 
The spine~$\spine{T}$ of a triangulation~$T$ of~$\ver{P}\loday$ is a binary tree. It is rooted by its only outgoing blossom, and oriented towards the root. Its internal nodes are labeled by inorder search labeling~\cite[p.~287]{CormenLeisersonRivestStein}, see \fref{fig:reconstruct}\,(right). It can therefore be seen as a \defn{binary search tree} with label set~$[n+1]$. Note that it is usual to either delete the blossoms as in~\cite{BjornerWachs} or to complete them to leaves/root as in~\cite{Loday}.
\end{example}

The following two immediate observations will be used later. Their proofs are left to the reader.

\begin{lemma}
\label{lem:RHS}
If~$\delta$ is a diagonal of a triangulation~$T$ of~$\ver{P}$, then the set~$\Abo(\delta)$ (resp.~$\Bel(\delta)$) labels the nodes of the connected component in~$\spine{T} \ssm \delta^*$ of the target (resp.~source) of~$\delta^*$. We refer to this connected component as the subtree of~$\spine{T}$ above~$\delta$ (resp.~below~$\delta$).
\end{lemma}

\begin{lemma}
\label{lem:crucial}
For any diagonal~$\delta$ of~$\ver{P}$ (internal or not), and any label~$j \in [n+1]$,
\begin{itemize}
\item if $j \in D \cap \Abo(\delta)$, then~$[j] \subseteq \Abo(\delta)$ or $[j,n+1] \subseteq \Abo(\delta)$,
\item if $j \in U \cap \Bel(\delta)$, then~$[j] \subseteq \Bel(\delta)$ or $[j,n+1] \subseteq \Bel(\delta)$.
\end{itemize}
\end{lemma}

Although we will not use it in this paper, let us observe that the spines of the triangulations of~$\ver{P}$ can be directly characterized as follows. This characterization relates our work to the recent preprint of K.~Igusa and J.~Ostroff~\cite{IgusaOstroff} who study these oriented and labeled trees in the perspective of cluster algebras and quiver representation theory.

\begin{proposition}
\label{prop:alternativeSpines}
A directed tree~$\spine{T}$ with blossoms and whose internal nodes are bijectively labeled by~$[n+1]$ is the spine of a triangulation of~$\ver{P}$ if and only if for all~$j \in [n+1]$:
\begin{itemize}
\item If~$j \in U$, then the node~$\nod{j}$ of~$\spine{T}$ labeled by~$j$ has indegree one and outdegree two, and all labels in the left outgoing subtree of~$\spine{T}$ at~$\nod{j}$ are smaller than~$j$, while all labels in the right outgoing subtree of~$\spine{T}$ at~$\nod{j}$ are larger than~$j$.
\item If~$j \in D$, then the node~$\nod{j}$ of~$\spine{T}$ labeled by~$j$ has indegree two and outdegree one, and all labels in the left incoming subtree of~$\spine{T}$ at~$\nod{j}$ are smaller than~$j$, while all labels in the right incoming subtree of~$\spine{T}$ at~$\nod{j}$ are larger than~$j$.
\end{itemize}
\end{proposition}

Note that we can reconstruct the labeling of a spine~$\spine{T}$ from its unlabeled plane structure together with the sets~$D$ and~$U$ as follows. We consider the permutation~$\sigma$ of~$[n+1]$ formed by all labels of~$D$ in increasing order followed by all labels of~$U$ in decreasing order. We traverse around the spine~$\spine{T}$ in counter-clockwise direction starting at the unique outgoing blossom that is followed by an incoming blossom. Whenever we encounter two incoming or two outgoing arcs at a node of~$\spine{T}$, we label this node by the next label of~$\sigma$.

\begin{exm}{1}{continued}
This procedure is illustrated in \fref{fig:reconstruct}\,(left) for the up and down label sets~$U\ex = \{2,3,5,7\}$ and $D\ex = \{1,4,6,8\}$. We traverse the tree counter-clockwise starting from the outgoing blossom marked by~$\#$. While traversing, we label each interior node where the direction of two consecutive edges changes by the next letter of~$\sigma\ex = [1,4,6,8,7,5,3,2]$. The resulting labeled oriented tree with blossoms is the spine~$\spine{(\widetilde T\ex)}$.
\end{exm}

\begin{exm}{2}{Loday's associahedron, continued}
For the polygon~$\ver{P}\loday$ where~$U\loday = \varnothing$ and $D\loday = [n+1]$, we have $\sigma\loday = [1, 2, \ldots, n+1]$, so that the procedure specializes to the inorder search labeling of binary trees~\cite[p.~287]{CormenLeisersonRivestStein}, which was used by J.-L.~Loday in~\cite{Loday}. This is illustrated in \fref{fig:reconstruct}\,(right).
\end{exm}

\begin{figure}[h]
  \centerline{\includegraphics[scale=1]{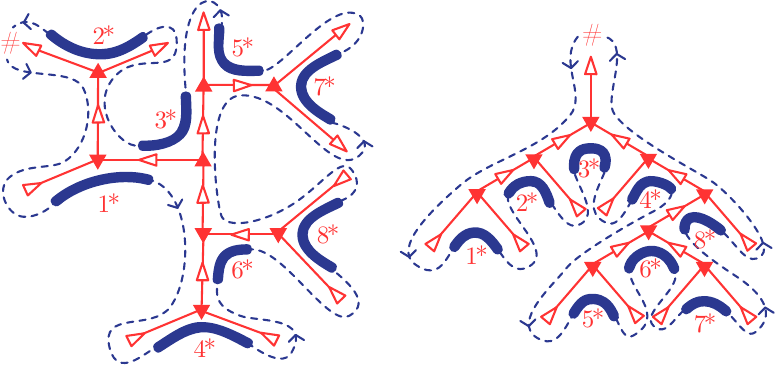}}
  \caption{Traverse in counter-clockwise direction around the unlabeled tree of spine~$\widetilde T\ex$ starting at blossom~$\#$ to reconstruct the inner node labels (left). The procedure specializes to inorder search labeling of binary trees (right).}
  \label{fig:reconstruct}
\end{figure}


\subsection{Associahedra}

The aim of this section is to state Theorem~\ref{theo:associahedron} which provides particular realizations~$\Asso$ of $n$-dimensional associahedra that are embedded in~$\R^{n+1}$ and parametrized by an $(n+3)$-gon~$\ver{P}$. More precisely, the realization~$\Asso$ does not depend on the geometry of~$\ver{P}$ but solely on the partition of~$[n+1]$ into up and down vertices of~$\ver{P}$. 

We first introduce points associated to triangulations of~$\ver{P}$ which are used in Theorem~\ref{theo:associahedron} to describe the vertices of~$\Asso$. Consider a triangulation~$T$ of~$\ver{P}$ with spine~$\spine{T}$. Let~$\Pi$ be the set of all \defn{undirected maximal paths} in~$\spine{T}$, that is, undirected paths connecting two blossoms of~$\spine{T}$. Note that a path~$\pi \in \Pi$ is not directed although each edge of~$\pi$ is oriented (as an arc of the spine~$\spine{T}$). For any~$j \in [n+1]$, we denote by~$\chor(j)$ the set of paths of~$\Pi$ whose edge orientation is reversed at node~$\nod{j}$. In other words, if~$j$ is a down (resp.~up) label, then~$\chor(j)$ is the set of paths of~$\Pi$ which use the two incoming (resp.~outgoing) arcs of~$\nod{j}$. It follows that~$|\chor(j)|$ is the product of the number of blossoms in the two incoming (resp.~outgoing) subtrees of~$\nod{j}$ in~$\spine{T}$. Associate to the triangulation~$T$ of~$\ver{P}$ the point~$\b{x}(T)\in \R^{n+1}$ with coordinates
\[
x_j(T) \eqdef \begin{cases} |\chor(j)| & \text{if } j \in D, \\ n+2-|\chor(j)| & \text{if } j \in U. \end{cases}
\]
Theorem~\ref{theo:associahedron} claims that the points~$\b{x}(T)$ associated to the triangulations~$T$ of~$\ver{P}$ are precisely the vertices of an associahedron~$\Asso$.

We now describe hyperplanes and half-spaces associated to internal diagonals of~$\ver{P}$ which are used in Theorem~\ref{theo:associahedron} to describe the inequalities of~$\Asso$. The interesting feature of the realization~$\Asso$ is that its facet defining inequalities are also facet defining inequalities of the permutahedron. Remember that the $n$-dimensional permutahedron~$\Perm$ is defined either as the convex hull of all permutations of~$[n+1]$ (seen as vectors in~$\R^{n+1}$), or as the intersection of the hyperplane~$H \eqdef H^=([n+1])$ with the half-spaces~$H^\ge(J)$ for~$J \subseteq [n+1]$, where
\[
H^=(J) \eqdef \biggset{\b{x} \in \R^{n+1}}{\sum_{j \in J} x_j = \binom{|J|+1}{2}}
\quad\text{and}\quad
H^\ge(J) \eqdef \biggset{\b{x} \in \R^{n+1}}{\sum_{j \in J} x_j \ge \binom{|J|+1}{2}}.
\]
To each internal diagonal or upper boundary edge~$\delta$ of the polygon~$\ver{P}$, we associate the hyperplane $H^=(\delta) \eqdef H^=(\Bel(\delta))$ and the half-space $H^\ge(\delta) \eqdef H^\ge(\Bel(\delta))$. Theorem~\ref{theo:associahedron} claims that the inequalities~$H^\ge(\delta)$ associated to internal diagonals~$\delta$ of~$\ver{P}$ are precisely the facet-defining inequalities for the associahedron~$\Asso$.

\begin{exm}{1}{continued}
\label{exm:small2}
The points associated to the triangulations~$T\ex$ and~$\widetilde T\ex$ of \fref{fig:flip} are~$\b{x}(T\ex) = (7,8,-6,1,7,10,8,1)$ and~$\b{x}(\widetilde T\ex) = (7,8,0,1,7,4,8,1)$. The half-spaces~$H^\ge(\ver{2}\ver{8})$ and~$H^\ge(\ver{3}\ver{9})$ are defined by the inequalities~$x_1 + x_2 + x_4 + x_6 \ge 10$ and~$x_1 + x_2 + x_3 + x_4 + x_6 + x_8 \ge 21$, respectively. Note that~$\b{x}(T\ex)$ and~$\b{x}(\widetilde T\ex)$ both belong to~$H$, $H^=(\ver{3}\ver{9})$ and~$H^\ge(\ver{2}\ver{8})$.
\end{exm}

\begin{exm}{2}{Loday's associahedron, continued}
\label{exm:Loday2}
For any triangulation~$T\loday$ of~$\ver{P}\loday$, the $i$\ordinal{} coordinate of~$\b{x}(T\loday)$ is the product of the number of leaves in the two children of the $i$\ordinal{} node of the dual binary tree of~$T\loday$ for the inorder search labeling (see \fref{fig:reconstruct}\,(right)). For~$i<j$, the half-space~$H^\ge(\ver{i}\ver{j})$ corresponding to a diagonal~$\ver{i}\ver{j}$ is given by the inequality~$\sum_{k=i+1}^{j-1} x_k \ge \binom{j-i}{2}$. See~\cite{Loday}.
\end{exm}

The following statements are the essence of the construction of~\cite{HohlwegLange}. Their proofs in terms of spines are deferred to Section~\ref{subsec:proofs_for_realization}.

\begin{proposition}[\protect{\cite[Prop.~2.10]{HohlwegLange}}]
\label{prop:main}
Let~$\delta$ be an internal diagonal or an upper boundary edge of~$\ver{P}$, and let~$T$ be a triangulation of~$\ver{P}$.
\begin{enumerate}[(i)]
\item The point~$\b{x}(T)$ is contained in the half-space~$H^\ge(\delta)$. 
\item The point~$\b{x}(T)$ lies on the hyperplane~$H^=(\delta)$ if and only if~$\delta$ belongs to~$T$.
\end{enumerate}
\end{proposition}

\begin{theorem}[\protect{\cite[Thms.~1.1 and~2.12]{HohlwegLange}}]
\label{theo:associahedron}
The following equivalent descriptions define an associahedron~$\Asso$:
\begin{enumerate}[(i)]
\item the convex hull of the points~$\b{x}(T)$ for all triangulations~$T$ of~$\ver{P}$,
\item the intersection of the hyperplane~$H$ with the half-spaces~$H^\ge(\delta)$ for all diagonals~$\delta$ of~$\ver{P}$.
\end{enumerate}
\end{theorem}


\subsection{Proof of Proposition~\ref{prop:main} and Theorem~\ref{theo:associahedron}}
\label{subsec:proofs_for_realization}
The proof of Proposition~\ref{prop:main} relies on a double counting argument. For a path~$\pi \in \Pi$, we denote by~$\chor(\pi)$ the set of labels~$j$ in~$[n+1]$ such that the edge orientation of~$\pi$ is reversed at node~$\nod{j}$. In other words, $j \in \chor(\pi)$ if and only if~$\pi \in \chor(j)$. We partition~$\chor(\pi)$ into the up labels~$\Up(\pi) \eqdef \chor(\pi) \cap U$ and the down labels~$\Down(\pi) \eqdef \chor(\pi) \cap D$. By definition, $\pi$ contains two outgoing arcs at the nodes labeled by~$\Up(\pi)$, two incoming arcs at the nodes labeled by~$\Down(\pi)$, and one incoming and one outgoing arc at all its other nodes. Therefore, the nodes labeled by~$\Up(\pi)$ and~$\Down(\pi)$ alternate along~$\pi$. Moreover, if we traverse along the path~$\pi$ from one endpoint~$\nod{s}$ to the other~$\nod{t}$, the first (resp.~last) node where the orientation reverses is labeled by~$\Up(\pi)$ if the blossom~$\nod{s}$ (resp.~$\nod{t}$) is outgoing and by~$\Down(\pi)$ if the blossom~$\nod{s}$ (resp.~$\nod{t}$) is incoming.

Using that~$\pi \in \chor(j)$ is equivalent to~$j \in \chor(\pi) = \Up(\pi) \sqcup \Down(\pi)$, we obtain by double counting
\begin{align*}
\sum_{j \in \Bel(\delta)} x_j(T) & = \!\! \sum_{d \in \Bel(\delta) \cap D} \!\! \big| \chor(d) \big| \, + \!\!\! \sum_{u \in \Bel(\delta) \cap U} \!\! \big( n+2-\big| \chor(u) \big| \big) \\
& = \up(\delta) (n+2) + \sum_{\pi \in \Pi} \big( \big| \Bel(\delta) \cap \Down(\pi) \big| - \big| \Bel(\delta) \cap \Up(\pi) \big| \big),
\end{align*}
where~$\up(\delta) \eqdef | \Bel(\delta) \cap U |$. The main idea of our approach is to show that each path~$\pi \in \Pi$ contributes to this sum at least as much as its \defn{minimal contribution}~$\mu(\pi)$ defined as
\begin{itemize}
\item $\mu(\pi) = 1$ if~$\pi$ connects two incoming blossoms below~$\delta$,
\item $\mu(\pi) = -1$ if~$\pi$ connects two outgoing blossoms below~$\delta$, or an outgoing blossom below~$\delta$ with any blossom above~$\delta$,
\item $\mu(\pi) = 0$ otherwise.
\end{itemize}
Summing the minimal contributions of all paths of~$\Pi$, we obtain that
\[
\sum_{j \in \Bel(\delta)} x_j(T) \ge \up(\delta) (n+2) + \binom{\bel(\delta)-\up(\delta)+1}{2} - \binom{\up(\delta)}{2} - \up(\delta) \big( n+2-\bel(\delta) \big) = \binom{\bel(\delta)+1}{2},
\]
since the spine~$\spine{T}$ has precisely~$\up(\delta)$ outgoing blossoms below~$\delta$, $\bel(\delta)-\up(\delta)+1$ incoming blossoms below~$\delta$, and~$n+2-\bel(\delta)$ blossoms above~$\delta$. It follows that~$\b{x}(T)$ is always contained in~$H^\ge(\delta)$. Moreover, $\b{x}(T)$ lies on~$H^=(\delta)$ if and only if the contri\-bution of each path~$\pi \in \Pi$ to the sum~$\sum_{j \in \Bel(\delta)} x_j(T)$ precisely coincides with the minimal contribution~$\mu(\pi)$ of~$\pi$.

The remaining of the proof discusses the contribution of each path to the sum~$\sum_{j \in \Bel(\delta)} x_j(T)$, depending on whether~$\delta$ belongs to~$T$ or not. Consider first the situation when~$\delta$ is a diagonal of~$T$. By Lemma~\ref{lem:RHS}, $\Bel(\delta)$ labels the subtree of the spine~$\spine{T}$ below~$\delta$. Thus, a maximal path~$\pi \in \Pi$ connecting two blossoms below~$\delta$ (resp.~above~$\delta$) completely remains below~$\delta$ (resp.~above~$\delta$), while a maximal path~$\pi \in \Pi$ connecting a blossom below~$\delta$ to a blossom above~$\delta$ stays below~$\delta$ until it leaves definitively~$\Bel(\delta)$ through the outgoing arc~$\delta^*$. Since the nodes labeled by~$\Down(\pi)$ and~$\Up(\pi)$ form an alternating sequence along~$\pi$ which starts and finishes with~$\Up(\pi)$ or~$\Down(\pi)$ depending on whether the endpoints of~$\pi$ are incoming or outgoing blossoms, the contribution ${\big| \Bel(\delta) \cap \Down(\pi) \big| - \big| \Bel(\delta) \cap \Up(\pi) \big|}$  of~$\pi$ to the sum~$\sum_{j \in \Bel(\delta)} x_j(T)$ is precisely its minimal contribution~$\mu(\pi)$. Therefore, if~$\delta$ is a diagonal of~$T$ then $\b{x}(T)$ belongs to the hyperplane~$H^=(\delta)$, and thus to the half-space~$H^\ge(\delta)$.

Consider now the situation when~$\delta$ is not a diagonal of~$T$. Fix a path~$\pi \in \Pi$ and the sequence~$S$ of the labels in~$\chor(\pi) \cap \Bel(\delta)$ along~$\pi$, which contribute to the sum~$\sum_{j \in \Bel(\delta)} x_j(T)$.
Note~that:
\begin{enumerate}[(i)]
\item $\pi$ always contributes at least~$-1$. Indeed, the sequence~$S$ cannot have two consecutive up labels, since the labels in~$\Down(\pi)$ and~$\Up(\pi)$ are alternating along~$\pi$, and since a down label between two labels below~$\delta$ is also below~$\delta$ by Lemma~\ref{lem:crucial}.
\item $\pi$ contributes at least~$1$ if its two endpoints are incoming blossoms below~$\delta$, and at least~$0$ if one of them~is. Indeed, if an endpoint of~$\pi$ is an incoming blossom below~$\delta$, and if the first orientation switch~$d \in \Down(\pi)$ along~$\pi$ is above~$\delta$, then the rest of the path~$\pi$ also lies above~$\delta$ by Lemma~\ref{lem:crucial}.
\item $\pi$ contributes at least~$0$ if its two endpoints are above~$\delta$. Indeed, if some~$u \in \Up(\pi)$ lies below~$\delta$, then one of the endpoints of~$\pi$ has to be below~$\delta$ by Lemma~\ref{lem:crucial}.
\end{enumerate}
Hence, the contribution of each path~$\pi \in \Pi$ to the sum~$\sum_{j \in \Bel(\delta)} x_j(T)$ is at least its minimal contribution~$\mu(\pi)$. As mentioned earlier, it follows that~$\b{x}(T)$ belongs to the half-space~$H^\ge(\delta)$ and it remains to prove that~$\b{x}(T)$ does not belong to~$H^=(\delta)$ if $\delta$ is not a diagonal of~$T$. It now suffices to find one path~$\pi \in \Pi$ whose contribution strictly exceeds its minimal contribution~$\mu(\pi)$. For this, consider any diagonal~$\gamma$ in~$T$ which crosses the diagonal~$\delta$. Such a diagonal exists since we are in the case where~$\delta$ is not a diagonal of~$T$. We claim that the contribution of any path~$\pi$ from a blossom~$\nod{s}$ above~$\delta$ but below~$\gamma$ to a blossom~$\nod{t}$ below~$\delta$ but above~$\gamma$ strictly exceeds its minimal contribution~$\mu(\pi)$. Since $\nod{t}$ is below~$\delta$ and $\nod{s}$ is above~$\delta$, we know that~$\mu(\pi)$ only
depends on~$\nod{t}$: $\mu(\pi)=0$ if $\nod{t}$ is incoming and~$\mu(\pi)=-1$ if $\nod{t}$ is outgoing. Using the observations (i) to (iii) above, it suffices to show that:
\begin{enumerate}[(i)]
\item If there is at least one label~$p \in \Bel(\delta)$ such that the edge orientation of~$\pi$ is reversed at node~$\nod{p}$, then the first such label while traversing from~$\nod{s}$ to~$\nod{t}$ along~$\pi$ must be a down~label. Then the contribution of~$\pi$ exceeds~$\mu(\pi)$ at least by this down label. 
\item Otherwise, the blossom~$\nod{t}$ must be outgoing, and~$\pi$ contributes~$0$ instead of~$\mu(\pi) = -1$.
\end{enumerate}
To prove~(i), suppose that there is a label~$p \in U \cap \Bel(\delta)$ such that~$\pi$ has two outgoing arcs at node~$\nod{p}$. 
Observe that~$p \in \Abo(\gamma)$ since~$\delta$ and~$\gamma$ cross and $p \in U$.
This implies that the arc~$\gamma^*$ must appear between~$\nod{s}$ and~$\nod{p}$ along~$\pi$, and is thus directed towards the node~$\nod{p}$. Therefore, the edge orientation of~$\pi$ must reverse between~$\gamma^*$ and~$\nod{p}$. This happens at a node~$\nod{q}$ for some~$q \in \Down(\pi)$. The label~$q$ must be below~$\delta$, otherwise Lemma~\ref{lem:crucial} would imply that either~$\gamma$ or~$p$ is above~$\delta$. In particular, $p$ is not the first label below~$\delta$ where the edge orientation of~$\pi$ is reversed and, as a consequence, the contribution of~$\pi$ to the sum exceeds~$\mu(\pi)$ by at least one.
To prove~(ii), assume that~$\nod{t}$ is incoming. Since the arc~$\gamma^*$ is directed towards~$\nod{t}$, the edge orientation of~$\pi$ must reverse at a node~$\nod{p}$ for some~$p \in \Down(\pi)$. Since~$\delta$ crosses~$\gamma$ and $\nod{t}$ is below~$\delta$, vertex~$\ver{p}$ is below~$\delta$ which contradicts $\Bel(\delta)\cap\chor(\pi)=\varnothing$. 
This concludes the proof of Proposition~\ref{prop:main}.

\bigskip
We now prove Theorem~\ref{theo:associahedron}. Consider the convex hull~$\conv \{\b{x}(T)\}$ of the points~$\b{x}(T)$, for all triangulations~$T$ of~$\ver{P}$. For any internal diagonal~$\delta$ of~$\ver{P}$, the hyperplane~$H^=(\delta)$ supports \emph{a priori} a face of~$\conv \{\b{x}(T)\}$. For a triangulation~$T$ of~$\ver{P}$ containing the diagonal~$\delta$, consider the triangulations~$T_1, \dots, T_{n-1}$ of~$\ver{P}$ obtained by flipping each internal diagonal~$\delta_1, \dots, \delta_{n-1}$ of~$T$ distinct from~$\delta$. The point~$\b{x}(T_i)$ is contained in all~$H^=(\delta_j)$ for~$j \ne i$ but not in~$H^=(\delta_i)$. Therefore, the points~$\b{x}(T), \b{x}(T_1), \dots, \b{x}(T_{n-1})$ are affinely independent, and they are all contained in the hyperplane~$H^=(\delta)$. The face supported by~$H^=(\delta)$ is thus a facet of~$\conv \{\b{x}(T)\}$. As each point~$\b{x}(T)$ is the intersection of at least~$n$ facets, it follows that~$\b{x}(T)$ is in fact a vertex of~$\conv \{\b{x}(T)\}$. The reverse inclusion poset of crossing-free sets of internal diagonals of~$\ver{P}$ thus injects in the inclusion poset of non-empty faces of~$\conv \{\b{x}(T)\}$, since the vertex-facets incidences are respected. Since both are face posets of $n$-dimensional polytopes, they must be isomorphic. As a consequence, $\conv \{\b{x}(T)\}$ is indeed an associahedron and its facet supporting hyperplanes are precisely the hyperplanes~$H^=(\delta)$, for all internal diagonals~$\delta$ of~$\ver{P}$. As stated in Theorem~\ref{theo:associahedron}, we therefore obtain an associahedron described by
\[
\Asso \; \eqdef \; \conv \set{\b{x}(T)}{T \text{ triangulation of } \ver{P}} \; = \; H \cap \!\!\!\!\!\! \bigcap_{\substack{\delta \text{ internal} \\ \text{diagonal of } \ver{P}}} \!\!\!\!\!\! H^\ge(\delta).
\]

\begin{remark}
Alternatively, we could prove Theorem~\ref{theo:associahedron} using Theorem~4.1 in~\cite{HohlwegLangeThomas} as follows. First, we construct the normal fan of the associahedron~$\Asso$ using Section~\ref{subsec:maximalNormalCones}. The conditions to apply Theorem~4.1 in~\cite{HohlwegLangeThomas} are then guaranteed by Lemma~\ref{lem:flip} and Proposition~\ref{prop:main}\,(ii) (the easier part of this proposition). In order to keep the presentation self-contained, we decided to provide direct proofs of Proposition~\ref{prop:main} and Theorem~\ref{theo:associahedron}.
\end{remark}


\section{Spines and flips}
\label{sec:flips}

In this section, we describe the behavior of spines when we perform flips in triangulations of~$\ver{P}$. For this, we first define spines of dissections of~$\ver{P}$ and their behavior under refinement. This leads to the crucial Lemma~\ref{lem:flip} which is implicitly used for various further results. The section ends with a discussion of relevant properties of the slope increasing flip lattice which is induced by a natural direction of flips. We emphasize that most statements are either immediate results or direct reformulations of results in~\cite{HohlwegLange} using the language of spines, and omit their proofs for this reason. Nevertheless, we reprove Lemma~\ref{lem:flip} to illustrate how spines simplify the arguments. 


\subsection{Dissections and their spines}

We first extend the definition of spines to \defn{dissections} of~$\ver{P}$, \ie crossing-free sets of internal diagonals of~$\ver{P}$ which decompose~$\ver{P}$ into polygonal~cells. Throughout this section, we call \defn{intermediate vertices} of a cell all its vertices except its leftmost and rightmost ones. The \defn{spine} of a dissection~$W$ of~$\ver{P}$ is its oriented and labeled dual tree~$\spine{W}$~with
\begin{itemize}
\item an internal node~$\nod{C}$ for each cell~$\ver{C}$ of~$W$, labeled by all the intermediate vertices of~$\ver{C}$,
\item an arc~$\delta^*$ for each internal diagonal~$\delta$ of~$W$ oriented from the cell below~$\delta$ to the cell above~$\delta$,
\item an outgoing (resp.~incoming) blossom~$\delta^*$ for each top (resp.~bottom) boundary edge~$\delta$ of~$\ver{P}$.
\end{itemize}
Observe that the outdegree (resp.~indegree) of a node~$\nod{C}$ is precisely the number of edges of the upper (resp.~lower) convex hull of~$\ver{C}$.

\begin{exm}{1}{continued}
\label{exm:small3}
\fref{fig:dissections} represents two dissections~$W\ex$ (left) and~$\widetilde W\ex$ (right). The nodes of the spine corresponding to cells with more than $3$ vertices are represented with squares~$\blacksquare$.

\begin{figure}[h]
  \centerline{\includegraphics[width=\textwidth]{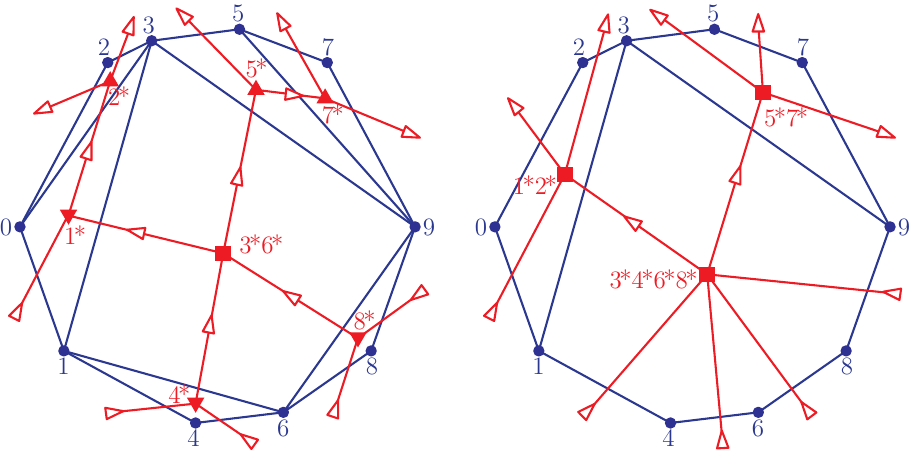}}
  \caption{Two dissections~$W\ex$ and~$\widetilde W\ex$ of the polygon~$\ver{P}\ex$ and their spines.}
  \label{fig:dissections}
\end{figure}
\end{exm}

\begin{exm}{2}{Loday's associahedron, continued}
\label{exm:Loday3}
The spine of any dissection of~$\ver{P}\loday$ is a tree naturally rooted by its only outgoing blossom, oriented towards its root, and where each node has at least~$2$ children. These trees are known as \defn{Schr\"oder trees}, see \eg~\cite{Stanley-SchroderTrees}.
\end{exm}

Let us mention that it is possible to characterize spines of dissections of~$\ver{P}$ in a similar way as we did for spines of triangulations of~$\ver{P}$ in Proposition~\ref{prop:alternativeSpines}. This extends the definition of~\cite{IgusaOstroff}.

\begin{proposition}
A directed tree~$\spine{W}$ with blossoms and whose internal nodes are labeled by subsets of~$[n+1]$ is the spine of a dissection of~$\ver{P}$ if and only if
\begin{enumerate}[(i)]
\item the labels of the nodes of~$\spine{W}$ form a partition of~$[n+1]$, and
\item at a node of~$\spine{W}$ labeled by~$X \subseteq [n+1]$, the source sets of the different incoming arcs belong to distinct intervals of~$[n+1] \ssm (X \cap D)$, while the sink sets of the different outgoing arcs belong to distinct intervals of~$[n+1] \ssm (X \cap U)$.
\end{enumerate}
\end{proposition}

To conclude on spines of dissections, we indicate how to completely reconstruct the spine of a dissection~$W$ from the collection of sets~$\Bel(\delta)$, for the diagonals~$\delta \in W$. We call \defn{directed line graph} of a spine~$\spine{W}$ the directed graph whose nodes correspond to the arcs of~$\spine{W}$ (excluding edges incident to blossoms), and with an arc between the nodes corresponding of~$\delta^*$ and~$\delta'^*$ if the head of~$\delta^*$ coincides with the source of~$\delta'^*$. Although interesting to motivate the introduction of spines, the following statement is not used later and its proof is left to the reader.

\begin{proposition}
The directed line graph of the spine~$\spine{W}$ of a dissection~$W$ of~$\ver{P}$ is the graph whose nodes correspond to the diagonals~$\delta \in W$ (excluding the boundary diagonals of~$\ver{P}$) and with an arc from~$\delta$ to~$\delta'$ if and only if $\Bel(\delta) \subset \Bel(\delta')$ and there is no diagonal~$\delta'' \in W$ such that
\[
\Bel(\delta) \subset \Abo(\delta'') \subset \Bel(\delta')
\quad\qquad\text{or}\qquad\quad
\Bel(\delta) \subset \Bel(\delta'') \subset \Bel(\delta').
\]
As a consequece, it is possible to reconstruct the unlabeled spine~$\spine{W}$ if~$\Bel(\delta)$ is known for each diagonal~$\delta \in W$. Moreover, the label set of a node~$\nod{C}$ of~$\spine{W}$ is given by
\[
\bigcap_{\delta_\Out} \Bel(\delta_\Out) \ssm \bigcup_{\delta_\In} \Bel(\delta_\In),
\]
where the intersection runs over the diagonals~$\delta_\Out$ corresponding to the outgoing arcs of~$\spine{W}$ at~$\nod{C}$ and the union runs over the diagonals~$\delta_\In$ corresponding to the incoming arcs of~$\spine{W}$ at~$\nod{C}$.
\end{proposition}

\begin{exm}{2}{Loday's associahedron, continued}
\label{exm:Loday4}
If~$W$ is a dissection of Loday's polygon~$\ver{P}\loday$, then the spine~$\spine{W}$ is isomorphic to the Hasse diagram of the poset formed by all diagonals of~$W$, where~$\delta < \delta'$ if~$\delta$ is below~$\delta'$. However, this coincidence only happens for Loday's polygon~$\ver{P}\loday$ (and for its reflection with respect to the horizontal axis): in general, the spine differs from its directed line graph.
\end{exm}


\subsection{Refinements and flips}
\label{subsec:refinement}

For two dissections~$W, \widetilde W$ of~$\ver{P}$, the cells of~$\widetilde W$ are unions of cells of~$W$ if and only if~$W \supseteq \widetilde W$ as sets of diagonals. We say that $W$ \defn{refines}~$\widetilde W$, and that~$\widetilde W$ \defn{coarsens}~$W$. The following statement is a direct consequence of the definitions, and is illustrated in \fref{fig:dissections}.

\begin{lemma}
\label{lem:refinement}
Consider two dissections~$W$ and~$\widetilde W$ of~$\ver{P}$ with $W\supseteq \widetilde W$. Then the spines~$\spine{W}$ and~$\spine{\widetilde W}$ can be constructed from each other as follows.
\begin{enumerate}[(i)]
\item Contracting the arcs~$\bigset{\delta^*}{\delta \in W \ssm \widetilde W}$ in the spine~$\spine{W}$ yields the spine~$\spine{\widetilde W}$.
\item Replacing each node~$\nod{C}$ of~$\spine{\widetilde W}$ by the spine of the dissection induced by~$W$ on the corresponding cell~$\ver{C}$ of~$\widetilde W$ yields $\spine{W}$.
\end{enumerate}
\end{lemma}

Each dissection~$W$ of~$\ver{P}$ corresponds to a face~$\face(W)$ of the associahedron~$\Asso$, and the inclusion poset of non-empty faces of the associahedron~$\Asso$ is (isomorphic to) the refinement poset of dissections of~$\ver{P}$. The next lemma gives the vertex and inequality descriptions of the~face~$\face(W)$.

\begin{lemma}[\protect{\cite[Cor.~2.11]{HohlwegLange}}]
\label{lem:faceDissection}
For any dissection~$W$ of~$\ver{P}$, the corresponding face~$\face(W)$ of the associahedron~$\Asso$ is given by
\[
\face(W) \; = \; \conv \set{\b{x}(T)}{\!\!\! \begin{array}{c} T \text{ triangulation of } \ver{P} \\ \text{that refines } W \end{array} \!\!\!} \; = \; H \cap \bigcap_{\delta \in W} H^=(\delta) \; \cap \bigcap_{\delta \vdash W} H^\ge(\delta),
\]
where~$\delta \vdash W$ denotes an internal diagonal~$\delta$ of~$\ver{P}$ compatible with~$W$ but not contained in~$W$, \ie which does not cross or coincide with any diagonal of~$W$. In the inequality description, each internal diagonal~$\delta \vdash W$ corresponds to a facet of~$\face(W)$.
\end{lemma}

Flips in triangulations are a key ingredient in the original proof of Proposition~\ref{prop:main} and Theorem~\ref{theo:associahedron} in~\cite{HohlwegLange}. Similarly, the evolution of spines while performing a flip will play an essential role in the remainder of this paper. The following statement extends~\cite[Lem.~2.6]{HohlwegLange} by relating flip directions explicitly to the oriented edges of a spine. As a consequence, every edge of the $1$-skeleton of~$\Asso$ gets a canonical orientation which will be used in Section~\ref{subsec:slopeIncreasingFlip}.

\begin{lemma}[\protect{\cite[Lem.~2.6]{HohlwegLange}}]
\label{lem:flip}
Let~$T$ and~$\widetilde T$ be two triangulations of~$\ver{P}$ related by a flip, let~$\delta \in T$ and $\tilde\delta \in \widetilde T$ be the two diagonals of~$\ver{P}$ such that~$T \ssm \delta = \widetilde T \ssm \tilde\delta$, and let~$i < j \in [n+1]$ label the two intermediate vertices of the quadrilateral with diagonals~$\delta$ and~$\tilde\delta$. If~$\delta^*$ is directed from~$\nod{i}$ to~$\nod{j}$, then~$\tilde\delta^*$ is directed from~$\nod{j}$ to~$\nod{i}$, and the difference~$\b{x}(\widetilde T) - \b{x}(T)$ is a positive multiple of~$e_i-e_j$.
\end{lemma}

\begin{proof}
\begin{figure}
  \centerline{
  \begin{overpic}[width=\textwidth]{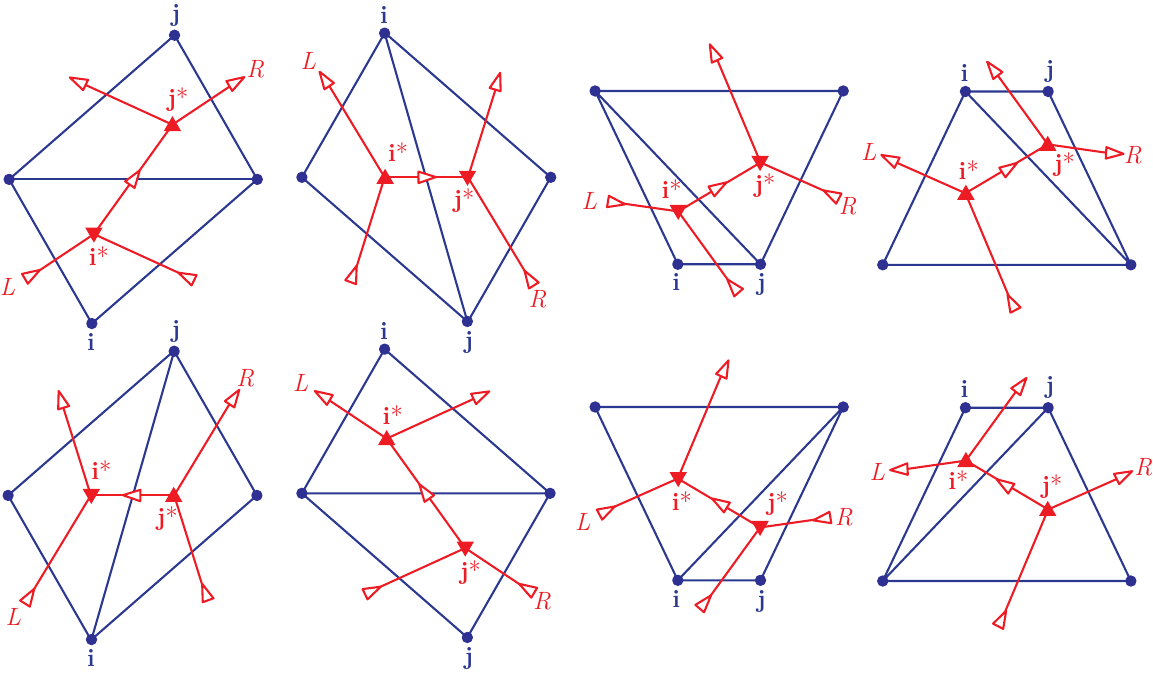}
		\put(7,43.5){\color{darkbluepic}$\delta$}
		\put(36,48){\color{darkbluepic}$\delta$}
		\put(57,45.5){\color{darkbluepic}$\delta$}
		\put(92.5,38){\color{darkbluepic}$\delta$}
		\put(11.5,21){\color{darkbluepic}$\tilde\delta$}
		\put(41,16.3){\color{darkbluepic}$\tilde\delta$}
		\put(67,18.3){\color{darkbluepic}$\tilde\delta$}
		\put(81,10){\color{darkbluepic}$\tilde\delta$}
		\put(-3,42.5){\color{darkbluepic}$T$}
		\put(-3,14.8){\color{darkbluepic}$\widetilde T$}
  \end{overpic}
  }
  \vspace{-.2cm}
  \caption{The four possible configurations of a square. Triangulations in the same column relate by a flip.}
  \label{fig:fourCases}
\end{figure}
Lemma~\ref{lem:refinement} ensures that the endpoints of~$\delta^*$ and~$\tilde\delta^*$ are necessarily~$\nod{i}$ and~$\nod{j}$. To see that $\delta^*$ and~$\tilde\delta^*$ have opposite orientations, it is enough to consider the case when~$\ver{P}$ is a square, again by Lemma~\ref{lem:refinement}. In this case, the property follows from a straightforward case analysis on the possible positions (up or down) of the vertices~$\ver{i}$ and~$\ver{j}$. The four different cases are illustrated in \fref{fig:fourCases}. Finally, in these four cases, we have~$\b{x}(\widetilde T) - \b{x}(T) = \alpha(e_i-e_j)$, where $\alpha$ is the product of the numbers of leaves in the subtrees of the spine~$\spine{T}$ labeled by~$L$ and~$R$ in \fref{fig:fourCases}.
\end{proof}

\begin{exm}{2}{Loday's associahedron, continued}
\label{exm:Loday5}
In Loday's polygon~$\ver{P}\loday$, we have seen that maximal spines are binary trees. The flip operation translates to the classical \defn{rotation} operation on binary trees~\cite[Section~3.2]{CormenLeisersonRivestStein}. Rotations were popularized when D.~Sleator, R.~Tarjan and W.~Thurston proved the diameter of the $n$-dimensional associahedron to be~$2n-4$ for sufficiently large~$n$~\cite{SleatorTarjanThurston}. L.~Pournin recently proved that this bound holds as soon as~$n > 9$~\cite{Pournin}.
\end{exm}


\subsection{Slope increasing flips}
\label{subsec:slopeIncreasingFlip}

By construction, the $1$-skeleton of the associahedron~$\Asso$ corresponds to the flip graph on the triangulations of~$\ver{P}$. In this section, we discuss relevant properties of this $1$-skeleton oriented along the direction~$\ver{U} \eqdef (n,n-2,\dots,-n+2,-n) \in \R^{n+1}$. The significance of this direction, in connection with the permutahedron and the weak order, will be discussed in Section~\ref{sec:normalFan}. For now, it is enough to observe that this orientation corresponds to slope increasing flips, as stated in the following lemma.

\begin{lemma}
\label{lem:slopeIncreasing}
Let~$T$ and~$\widetilde T$ be two triangulations of~$\ver{P}$ related by a flip, let~$\delta \in T$ and~$\tilde \delta \in \widetilde T$ be the two diagonals of~$\ver{P}$ such that~$T \ssm \delta = \widetilde T \ssm \tilde \delta$, and let~$i < j \in [n+1]$ label the two intermediate vertices of the quadrilateral with diagonals~$\delta$ and~$\tilde \delta$. Then the following assertions are equivalent.
\begin{enumerate}[(i)]
\item The vector~$\b{x}(\widetilde T) - \b{x}(T)$ is in the direction of~$\b{U}$, that is, $\dotprod{\b{U}}{\b{x}(\widetilde T) - \b{x}(T)} > 0$.
\item The arc~$\delta^*$ of~$\spine{T}$ is directed from~$\nod{i}$ to~$\nod{j}$, and the arc~$\tilde \delta^*$ of~$\spine{\widetilde T}$ is directed from~$\nod{j}$ to~$\nod{i}$.
\item The slope of~$\delta$ is smaller than the slope of~$\tilde \delta$.
\end{enumerate}
\end{lemma}

If these conditions are satisfied, we say that the flip from~$T$ to~$\widetilde T$ is \defn{slope increasing}. For example, all flips of \fref{fig:fourCases}, oriented from the upper triangulation to the lower one, are slope increasing flips. We consider the directed graph of slope increasing flips~$\flipGraph(\ver{P})$, with one node for each triangulation of~$\ver{P}$, and one arc between two triangulations related by a slope increasing flip. Lemma~\ref{lem:slopeIncreasing} can be reformulated as follows.
 
\begin{corollary}[\protect{\cite[Thm.~6.2]{Reading-CambrianLattices}}]
The graph~$\flipGraph(\ver{P})$ of slope increasing flips on the triangulations of~$\ver{P}$ is (isomorphic to) the $1$-skeleton of the associahedron~$\Asso$ oriented by the linear functional~${\b{x} \mapsto \dotprod{\b{U}}{\b{x}}}$.
\end{corollary}

The oriented graph~$\flipGraph(\ver{P})$ is clearly acyclic, and its transitive closure is known to be a lattice. It is an example of a \defn{Cambrian lattice} of type~$A$~\cite{Reading-CambrianLattices}, see Section~\ref{sec:existingWork}. The minimal and maximal elements of this lattice can be easily described as follows. The corresponding triangulations, deprived of the lattice of slope increasing flips, appear already in~\cite[Lem.~2.2 and Cor.~2.3]{HohlwegLange}.

\begin{proposition}
\label{prop:greedy}
The slope increasing flip graph~$\flipGraph(\ver{P})$ has a unique source~$\minTriang(\ver{P}) \eqdef \{\minDiag^{\ver{P}}_1, \dots, \minDiag^{\ver{P}}_n\}$ and a unique sink~$\maxTriang(\ver{P}) \eqdef \{\maxDiag^{\ver{P}}_1, \dots, \maxDiag^{\ver{P}}_n\}$, where the endpoints of the diagonal~$\minDiag^{\ver{P}}_i$ are labeled by
\[
\max \big( \{0\} \cup (U \cap [i]) \big) \quad\text{and}\quad \min \big( (D \ssm [i]) \cup \{n+2\} \big),
\]
while the endpoints of the diagonal~$\maxDiag^{\ver{P}}_i$ are labeled by
\[
\min \big( (U \ssm [n+1-i]) \cup \{n+2\} \big) \quad\text{and}\quad \max \big( \{0\} \cup (D \cap [n+1-i]) \big).
\]
The associated vertices of~$\Asso$ are~$\b{x}(\minTriang(\ver{P})) = (1, 2, \dots, n+1)$ and~$\b{x}(\maxTriang(\ver{P})) = (n+1, n, \dots, 1)$. The spines~$\spine{\minTriang(\ver{P})}$ and~$\spine{\maxTriang(\ver{P})}$ are directed and label-monotone paths with blossoms. Moreover, the collections of sets~$\Bel(\minDiag^{\ver{P}}_i) = [i]$, for~$i \in [n]$, and~$\Bel(\maxDiag^{\ver{P}}_i) = [n+2-i, n+1]$, for~$i \in [n]$, ordered by inclusion are both linear orders.
\end{proposition}

\begin{exm}{1}{continued}
\label{exm:small4}
In \fref{fig:flip}, the flip from the triangulation~$T\ex$ (left) to the triangulation~$\widetilde T\ex$ (right) is increasing. \fref{fig:greedy} represents the source~$\minTriang(\ver{P}\ex)$ and the sink~$\maxTriang(\ver{P}\ex)$ of the slope increasing flip graph~$\flipGraph(\ver{P}\ex)$ for the polygon~$\ver{P}\ex$ of Example~\ref{exm:small1}.

\begin{figure}
  \centerline{\includegraphics[width=\textwidth]{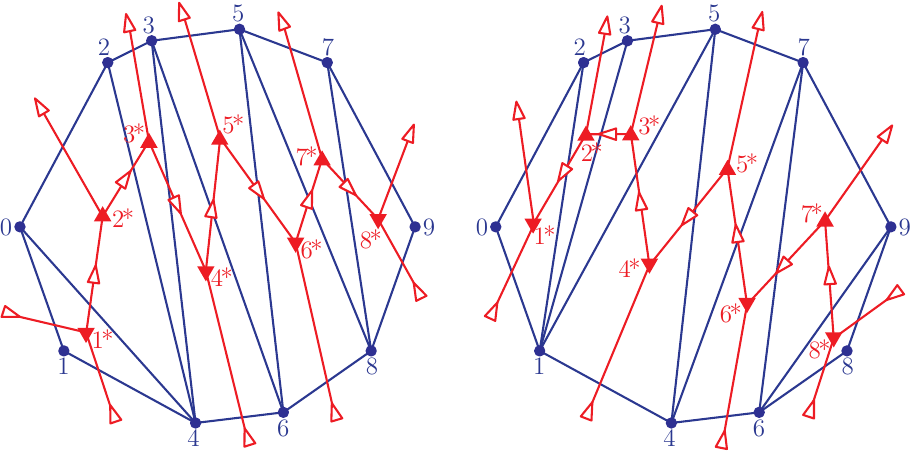}}
  \vspace{-.2cm}
  \caption{The source~$\minTriang(\ver{P}\ex)$ and the sink~$\maxTriang(\ver{P}\ex)$ of the slope increasing flip graph~$\flipGraph(\ver{P}\ex)$.}
  \label{fig:greedy}
\end{figure}
\end{exm}

\begin{exm}{2}{Loday's associahedron, continued}
\label{exm:Loday6}
For Loday's polygon~$\ver{P}\loday$, the transitive closure of the graph of slope increasing flips is the famous \defn{Tamari lattice}~\cite{TamariFestschrift}.
The minimal and maximal triangulations~$\minTriang(\ver{P})$ and~$\maxTriang(\ver{P})$ are the fan triangulations whose internal diagonals are all connected to the vertices~$\b{0}$ and~$\b{n+2}$ respectively, and the corresponding spines~$\spine{\minTriang(\ver{P})}$ and~$\spine{\maxTriang(\ver{P})}$ are left and right combs respectively.
\end{exm}

Further properties of the slope increasing flip lattice, in connection with the weak order on the permutations of~$[n+1]$ are discussed later in Section~\ref{sec:latticeQuotient}.


\section{Spines and normal fans}
\label{sec:normalFan}

This section focusses on the normal fan of the associahedron~$\Asso$. The main observation is that the normal cone of a face of~$\Asso$ is described by the spine of that face. In particular, we explore the connection between the permutahedron~$\Perm$, the associahedron~$\Asso$, and the parallelepiped~$\Para$ formed by the facets of~$\Perm$ incident to the vertices~$(1, 2, \dots, n+1)$ or~$(n+1, n, \dots, 1)$. The normal fans of these polytopes are successive coarsenings, thus inducing natural surjections from the permutations of~$[n+1]$ to the triangulations of~$\ver{P}$, and from the triangulations of~$\ver{P}$ to the binary words of length~$n$. We describe these surjections combinatorially in terms of spines. Essentially all statements in this section are combinations of Lemma~\ref{lem:flip} with results of the dictionary between preposets and braid cones discussed by A.~Postnikov, V.~Reiner and L.~Williams in~\cite[Sect.~3.4]{PostnikovReinerWilliams}.


\subsection{Maximal normal cones}
\label{subsec:maximalNormalCones}

Using spines, we can describe the normal cones of any face of the associahedron~$\Asso$. Nevertheless, we first concentrate on vertices and describe maximal normal cones using triangulations and their spines. We discuss the more general setting in Section~\ref{subsec:normalFans} where we recast Propositions~\ref{prop:maximalNormalCones},~\ref{prop:maximalCambrianCones} and~\ref{prop:singletons} into Propositions~\ref{prop:normalFan},~\ref{prop:CambrianFan} and~\ref{prop:commonFaces}. 

For each vertex~$\b{v}$ of a given polytope~$P \subset \R^n$, the \defn{cone~$C(\b{v})$ of~$\b{v}$} and the \defn{normal cone~$C\polar(\b{v})$ of~$\b{v}$} are the polyhedral cones
\[
C(\b{v}) \eqdef \cone \bigset{\b{v'}-\b{v}}{\b{v'} \text{ vertex of } P}
\quad\text{and}\quad
C\polar(\b{v}) \eqdef \cone \bigset{\b{u} \in \R^n}{\dotprod{\b{u}}{\b{v}} = \max\nolimits_{\b{x} \in P} \, \dotprod{\b{u}}{\b{x}}}.
\]
As the notation suggests, these two cones are polar to each other. Note that we consider the normal cones of the polytopes~$\Asso$ and~$\Perm$ embedded in their affine span~$H \eqdef H^=([n+1])$. From Lemma~\ref{lem:flip}, we obtain a description of both cones~$C(\b{v})$ and~$C\polar(\b{v})$ for the associahedron~$\Asso$.

\begin{proposition}[Lemma~\ref{lem:flip} and~\protect{\cite[Prop.~3.5]{PostnikovReinerWilliams}}]
\label{prop:maximalNormalCones}
For any triangulation~$T$ of~$\ver{P}$, the cone~$C(T)$ of the vertex~$\b{x}(T)$ of the associahedron~$\Asso$ is the \defn{incidence cone} of the spine~$\spine{T}$ while the normal cone~$C\polar(T)$ of~$\b{x}(T)$ is the \defn{braid cone} of the spine~$\spine{T}$, that is,
\[
C(T) = \cone \set{e_i-e_j}{\nod{i}\nod{j} \in \spine{T}}
\qquad\text{and}\qquad
C\polar(T) = \bigcap_{\nod{i}\nod{j} \in \spine{T}} \set{\b{u} \in H}{u_i \le u_j},
\]
where $\nod{i}\nod{j} \in \spine{T}$ means that~$\nod{i}\nod{j}$ is an oriented arc of~$\spine{T}$.
\end{proposition}

We now recall basic properties of the permutahedron. The normal fan of the permutahedron~$\Perm \subset \R^{n+1}$ is the fan defined by the \defn{braid arrangement} which is formed by the hyperplanes~$x_i=x_j$, for~$i \ne j \in [n+1]$. For each permutation~$\sigma$ of~$[n+1]$, the maximal cone~$C\polar(\sigma) \eqdef \set{\b{x} \in H}{x_{\sigma^{-1}(i)} \le x_{\sigma^{-1}(i+1)} \text{ for } i \in [n]}$ of the braid arrangement is the normal cone of the vertex~$(\sigma(1),\dots,\sigma(n+1))$ of~$\Perm$.

Given a triangulation~$T$ of~$\ver{P}$, its spine~$\spine{T}$ is an acyclic graph labeled by~$[n+1]$ and its transitive closure defines a partial order~$\prec_T$ on~$[n+1]$. Remember that a \defn{linear extension} of~$\prec_T$ is a linear order~$\prec_L$ such that~$i \prec_T j$ implies~$i \prec_L j$. Equivalently, a linear extension of~$\prec_T$ can be seen as a permutation~$\sigma$ of~$[n+1]$ such that~$i \prec_T j$ implies~$\sigma(i) < \sigma(j)$. Note that the linear order~$\prec_L$ then coincides with the one-line notation of the inverse of~$\sigma$. These linear extensions precisely determine the normal cones~$C\polar(\sigma)$ of~$\Perm$ which are contained in the normal cone~$C\polar(T)$ of~$\Asso$~as~follows.

\begin{proposition}[\protect{\cite[Cor.~3.9]{PostnikovReinerWilliams}}]
\label{prop:maximalCambrianCones}
For any triangulation~$T$ of~$\ver{P}$, the normal cone~$C\polar(T)$ in~$\Asso$ is the union of the normal cones~$C\polar(\sigma)$ in~$\Perm$ of all linear extensions~$\sigma$ of the transitive closure~$\prec_T$ of~$\spine{T}$.
\end{proposition}

Proposition~\ref{prop:maximalCambrianCones} naturally defines a surjective map~$\surjectionPermAsso$ from the permutations of~$[n+1]$ to the triangulations of~$\ver{P}$, which sends a permutation~$\sigma$ of~$[n+1]$ to the unique triangulation~$T$ of~$\ver{P}$ such that~$C\polar(T)$ contains~$C\polar(\sigma)$. By Proposition~\ref{prop:maximalCambrianCones}, the fiber~$\surjectionPermAsso^{-1}(T)$ of a triangulation~$T$ of~$\ver{P}$ is the set of linear extensions of the spine~$\spine{T}$.

\begin{exm}{1}{continued}
\label{exm:small5}
Consider the triangulation~$T\ex$ of \fref{fig:flip}\,(left). Its fiber~$\surjectionPermAsso^{-1}(T\ex)$ contains $35$ permutations, for example~$[3,4,2,1,7,6,8,5]$, $[4,6,2,1,7,5,8,3]$, and~$[7,8,3,2,5,4,6,1]$, since their corresponding linear orders (\ie their inverse permutations) extend~$T\ex$.
\end{exm}

N.~Reading~\cite{Reading-CambrianLattices} described combinatorially the surjection~$\surjectionPermAsso$ as follows (see \fref{fig:paths}). Fix a permutation~$\sigma$ of~$[n+1]$. For~$i \in \{0,\dots,n+1\}$, define~$\pi_i(\sigma)$ to be the $x$-monotone path in~$\ver{P}$ joining the vertices~$\ver{0}$ and~$\ver{n+2}$ and passing through the vertices of the symmetric difference~$D \symdif \sigma^{-1}([i])$, with the convention that~$[0] = \varnothing$. As illustrated in \fref{fig:paths}, the sequence of paths~$\pi_0(\sigma), \pi_1(\sigma), \dots, \pi_{n+1}(\sigma)$ sweeps the polygon~$\ver{P}$, starting from the lower hull of~$\ver{P}$ and ending at the upper hull of~$\ver{P}$. The path~$\pi_i(\sigma)$ differs from the path~$\pi_{i-1}(\sigma)$ by a single vertex: namely, vertex~$\sigma^{-1}(i)$ belongs to $\pi_{i-1}(\sigma)$ but not to~$\pi_i(\sigma)$ if it is a down vertex, and to $\pi_i(\sigma)$ but not to~$\pi_{i-1}(\sigma)$ if it is an up vertex. The triangulation~$\surjectionPermAsso(\sigma)$ associated to the permutation~$\sigma$ is the union of the paths~$\pi_0(\sigma), \pi_1(\sigma), \dots, \pi_{n+1}(\sigma)$.

\begin{figure}
  \centerline{\includegraphics[scale=.6]{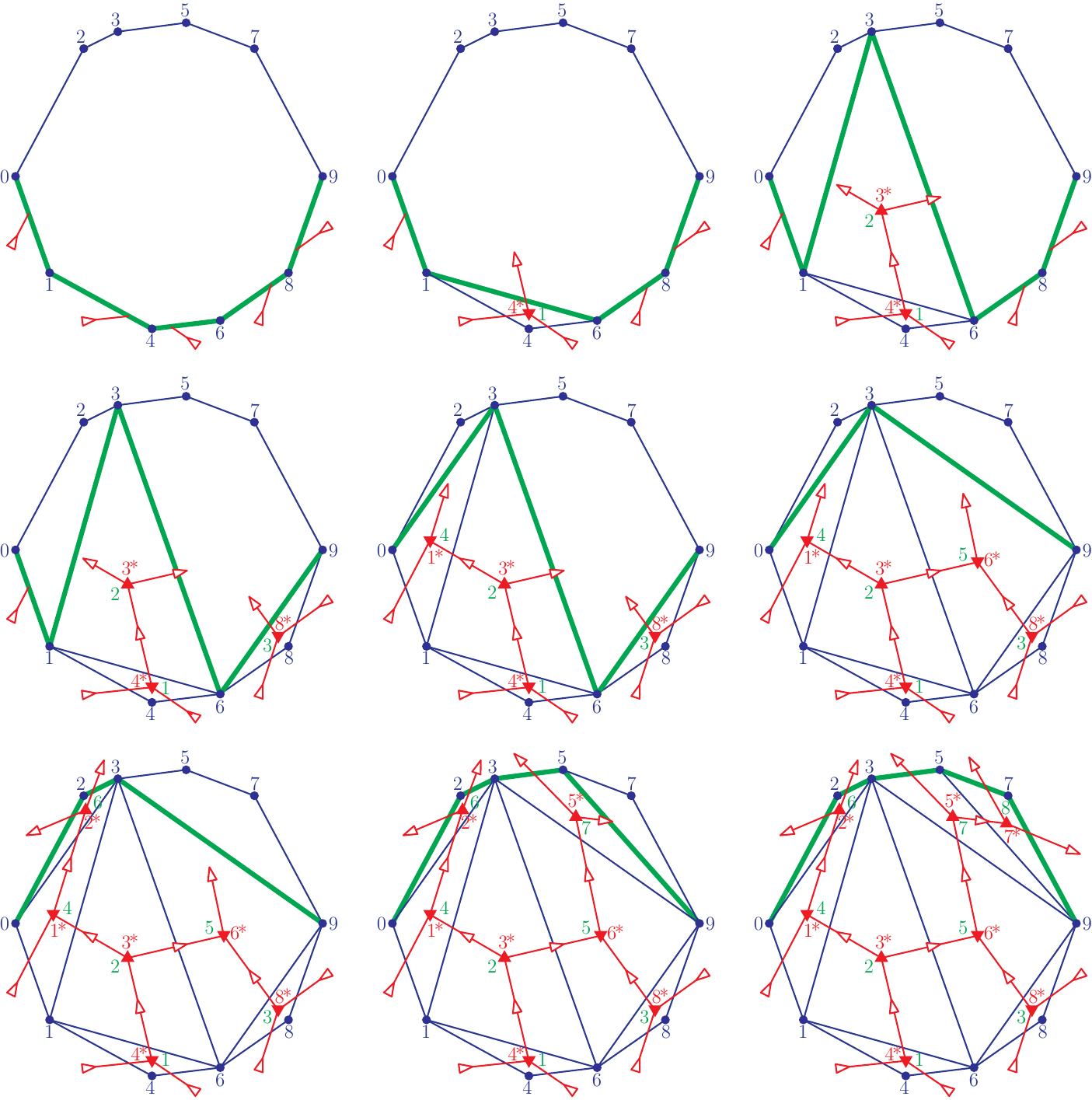}}
  \caption{The sequence of paths~$\pi_0(\sigma\ex), \pi_1(\sigma\ex), \dots, \pi_8(\sigma\ex)$ for~$\sigma\ex = [4,6,2,1,7,5,8,3]$.}
  \label{fig:paths}
\end{figure}
\begin{figure}
  \centerline{\includegraphics[scale=1]{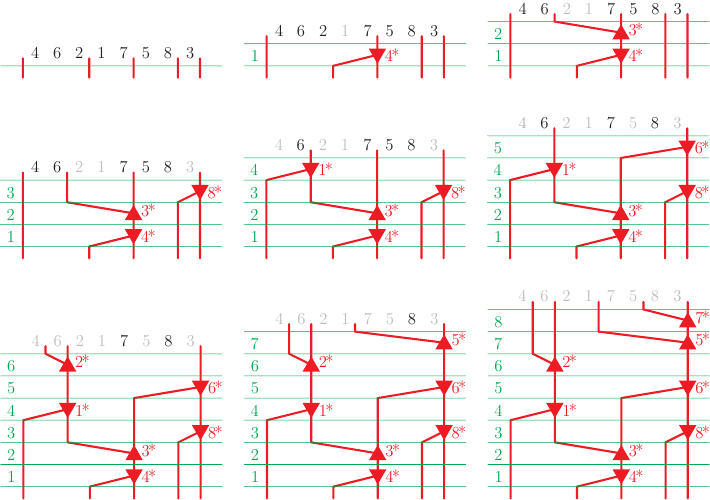}}
  \caption{The construction of the leveled spine~$\leveledBijection{\sigma\ex}$ for~$\sigma\ex = [4,6,2,1,7,5,8,3]$.}
  \label{fig:leveledSpines}
\end{figure}

This surjection can also be described directly in terms of spines (see \fref{fig:leveledSpines}). A \defn{leveled spine}~$\leveledSpine{T}$ is a spine~$\spine{T}$ whose nodes are additionally ordered by one of the linear extensions of~$\spine{T}$. In other words, the nodes of~$\leveledSpine{T}$ are ordered by levels such that the partial order forced by the arcs of~$\spine{T}$ is respected. From a permutation~$\sigma$ of~$[n+1]$, we can directly construct a leveled spine~$\leveledBijection{\sigma}$ whose underlying spine (when forgetting the levels) is the spine~$\spine{\surjectionPermAsso(\sigma)}$ of the triangulation~$\surjectionPermAsso(\sigma)$ associated to the permutation~$\sigma$. The following algorithm constructs~$\leveledBijection{\sigma}$ level by level from bottom to top. We start writing~$\sigma$ in one-line notation~$\sigma(1)\sigma(2) \dots \sigma(n+1)$, and place a strand~$\pmb{|}$ before each position~${d \in D}$ as well as after the last position. At each step~$i$, we add a new node at level~$i$ labeled by~$\sigma^{-1}(i)$ as follows:
\begin{itemize}
\item If~$\sigma^{-1}(i)$ is a down label, then the node at level~$i$ is a node~$\blacktriangledown$, with two incoming and one outgoing strands. The strand just to the left of~$i$ glides to the right until it meets the next strand at~$\blacktriangledown$ where they merge to form a new vertical strand. See \fref{fig:leveledSpines}, Steps 1, 3,~4,~and~5.
\item If~$\sigma^{-1}(i)$ is an up label, then the node at level~$i$ is a node~$\blacktriangle$, with one incoming and two outgoing strands. The first strand to the right of~$i$ splits at~$\blacktriangle$ into two strands; one remains vertical while the other glides to the left and stops just left of~$i$. See \fref{fig:leveledSpines}, Steps 2, 6, 7, and~8.
\end{itemize}
The resulting leveled tree~$\leveledBijection{\sigma}$, where the node at level $i$ is labeled by~$\sigma^{-1}(i)$, and where the arcs are oriented upwards, gives the spine~$\spine{\surjectionPermAsso(\sigma)}$ when we forget the levels. Note also that we could have described this construction from top to bottom, but we have chosen our presentation to match the construction of the triangulation~$\surjectionPermAsso(\sigma)$ described by N.~Reading in~\cite{Reading-CambrianLattices}. An alternative presentation, where the nodes of~$\leveledBijection{\sigma}$ give the table of the permutation~$\sigma$ can be found in~\cite{ChatelPilaud}.

\begin{exm}{1}{continued}
\label{exm:small6}
Consider the permutation~$\sigma\ex \eqdef [4,6,2,1,7,5,8,3]$, given by its one-line notation. The corresponding linear order is given by~$(\sigma\ex)^{-1} = [4,3,8,1,6,2,5,7]$. \fref{fig:paths} illustrates the successive steps to construct the triangulation~$\surjectionPermAsso(\sigma\ex)$, while \fref{fig:leveledSpines} shows the successive steps to obtain the leveled spine~$\leveledBijection{\sigma\ex}$. The resulting triangulation and spine are those of \fref{fig:flip}\,(left). 
\end{exm}

\begin{exm}{2}{Loday's associahedron, continued}
\label{exm:Loday7}
In the situation of Loday's polygon~$\ver{P}\loday$, the surjection~$\surjectionPermAsso$ can also be described in terms of \defn{insertion} in binary search trees~\cite[Section~12.3]{CormenLeisersonRivestStein}. Namely, for any permutation~$\sigma$ of~$[n+1]$, the spine~$\surjectionPermAsso(\sigma)$ is the last tree in the sequence of binary search trees~$\varnothing \eqdef \tau_0 \subset \tau_1 \subset \dots \subset \tau_{n+1} \defeq \surjectionPermAsso(\sigma)$, where~$\tau_i$ is obtained by insertion of~$\sigma^{-1}(n+2-i)$ in~$\tau_{i-1}$.
\end{exm}

Of particular interest are triangulations~$T$ of~$\ver{P}$ whose fiber is a singleton. The corresponding points~$\b{x}(T)$ are precisely the vertices of the associahedron~$\Asso$ that are also vertices of the permutahedron~$\Perm$. Moreover, they determine the half-spaces of the permutahedron that are needed to construct~$\Asso$: a facet-defining inequality of the permutahedron is a facet-defining inequality of~$\Asso$ if and only if its bounding hyperplane contains a point~$\b{x}(T)$ where $\surjectionPermAsso^{-1}(T)$ is a singleton. This property was used in~\cite{HohlwegLangeThomas} to generalize the construction of~$\Asso$ to generalized associahedra for any finite Coxeter group as indicated in Section~\ref{sec:existingWork}. Relevant examples of triangulations whose fiber is a singleton are the minimal and maximal triangulations~$\minTriang(\ver{P})$ and~$\maxTriang(\ver{P})$ in the slope increasing flip lattice on~$\ver{P}$ as presented in Proposition~\ref{prop:greedy}. More generally, the following equivalent statement characterizes triangulations whose fiber is a singleton. It is a slight extension of~\cite[Prop.~1.4]{HohlwegLange}, translated in terms of spines.

\begin{proposition}[\protect{\cite[Prop.~1.4]{HohlwegLange}}]
\label{prop:singletons}
For a triangulation~$T$ of~$\ver{P}$ and a permutation~$\sigma$ of~$[n+1]$, the following assertions are equivalent:
\begin{enumerate}[(i)]
\item The spine~$\spine{T}$ is a directed path (with blossoms) labeled by~$\sigma$.
\item The transitive closure~$\prec_T$ of the spine~$\spine{T}$ is the linear order defined by~$\sigma$.
\item The fiber of~$T$ with respect to the map~$\surjectionPermAsso$ is the singleton~$\surjectionPermAsso^{-1}(T) = \{\sigma\}$.
\item The vertex~$\b{x}(T)$ of~$\Asso$ coincides with the vertex~$(\sigma(1), \dots, \sigma(n+1))$ of~$\Perm$.
\item The cone~$C(T)$ of~$\Asso$ coincides with the cone~$C(\sigma)$ of~$\Perm$.
\item The normal cone~$C\polar(T)$ of~$\Asso$ coincides with the normal cone~$C\polar(\sigma)$ of~$\Perm$.
\end{enumerate}
\end{proposition}

\medskip
We now consider the parallelepiped~$\Para$ defined as the intersection of the hyperplane $H \eqdef H^=([n+1])$ with the half-spaces~$H^\ge([i])$ and~$H^\ge([n+1] \ssm [i])$ for~$i \in [n]$. In other words, $\Para$ is defined by $H \eqdef H^=([n+1])$ and the facet-defining inequalities of~$\Perm$ whose bounding hyperplane contains~$(1, 2, \dots, n+1)$ or~$(n+1, n, \dots, 1)$. We label the vertices of~$\Para$ by binary words of length~$n$ on the alphabet~$\{\pm 1\}$ such that a vertex labeled by~$w \eqdef w_1 \cdots w_n \in \{\pm 1\}^n$ belongs to~$H^\ge([i])$ if~$w_i = 1$ and to~$H^\ge([n+1] \ssm [i])$ if~$w_i = -1$. The normal cone of the vertex~$w$ of~$\Para$ is the polyhedral cone~$C\polar(w) \eqdef \set{\b{x} \in H}{w_i(x_i - x_{i+1}) \le 0 \text{ for } i \in [n] }$. The following statement characterizes which maximal normal cones of~$\Asso$ are contained in~$C\polar(w)$.

\begin{proposition}
\label{prop:parallelepiped}
In any spine~$\spine{T}$, the nodes~$\nod{i}$ and~$\nod{(i+1)}$ are always comparable. For any binary word~$w \in \{\pm 1\}^n$, the normal cone~$C\polar(w)$ in~$\Para$ is the union of the normal cones~$C\polar(T)$ in~$\Asso$ of all triangulations~$T$ of~$\ver{P}$ with the following property: for all~$i \in [n]$, the node~$\nod{i}$ is below~$\nod{(i+1)}$ in~$\spine{T}$ if~$w_i = 1$ and above~$\nod{(i+1)}$ in~$\spine{T}$ if~$w_i = -1$.
\end{proposition}

Proposition~\ref{prop:parallelepiped} naturally defines a surjective map~$\surjectionAssoPara$ from the triangulations of~$\ver{P}$ to the binary words of~$\{\pm 1\}^n$, which sends a triangulation~$T$ of~$\ver{P}$ to the unique word~$w \in \{\pm 1\}$ such that~$C\polar(w)$ contains~$C\polar(T)$.

\begin{exm}{2}{Loday's associahedron, continued}
\label{exm:Loday8}
J.-L.~Loday describes in~\cite[Sect.~2.7]{Loday} an equivalent surjection related to his realization. For a triangulation~$T$ of~$\ver{P}\loday$, the image~$\surjectionAssoPara(T)$ is called the \defn{canopy} of the binary tree~$\spine{T}$. Alternative equivalent definitions are possible, see~\cite{Viennot}. The reader is invited to work out similar definitions for a general polygon~$\ver{P}$.
\end{exm}

Note that composing~$\surjectionPermAsso$ with $\surjectionAssoPara$ gives the \defn{recoil} map on permutations: $\surjectionAssoPara(\surjectionPermAsso(\sigma)) = w_1 \cdots w_n$ where~$w_i = 1$ if~$\sigma_i < \sigma_{i+1}$ and $w_i = -1$ otherwise.


\subsection{Normal fans}
\label{subsec:normalFans}

We now extend the results of the previous section to completely describe the normal fan of~$\Asso$. In particular, Propositions~\ref{prop:normalFan},~\ref{prop:CambrianFan} and~\ref{prop:commonFaces} extend Propositions~\ref{prop:maximalNormalCones},~\ref{prop:maximalCambrianCones} and~\ref{prop:singletons} from facets to any $k$-dimensional face.  In other words, we describe the geometry of the corresponding Cambrian fan of type~$A$~\cite{ReadingSpeyer}, see Section~\ref{sec:existingWork}. The \defn{normal cone}~$C\polar(F)$ of a $k$-dimensional face~$F$ of an $n$-dimensional polytope~$P \subset \R^n$ is defined~as
\[
C\polar(F) \eqdef \cone \bigset{\b{u} \in \R^n}{\dotprod{\b{u}}{\b{v}} = \max\nolimits_{\b{x} \in P} \, \dotprod{\b{u}}{\b{x}}, \text{ for all } \b{v} \in F}
\]
and has dimension~${n-k}$. The \defn{normal fan} of~$P$ is the complete fan formed by the normal cones of all faces of~$P$. To describe the normal cones of all faces of~$\Perm$ and of~$\Asso$, we need to introduce the following notions. For a more detailed presentation, we refer to the dictionary between preposets and braid cones  presented in~\cite[Sect.~3.4]{PostnikovReinerWilliams}. 

A \defn{preposet} on~$[n+1]$ is a binary relation~$R \subseteq [n+1] \times [n+1]$ which is reflexive and transitive. Hence, any equivalence relation is a symmetric preposet and any poset is an antisymmetric preposet.
Any preposet~$R$ can in fact be decomposed into an equivalence relation~${\equiv_R} \eqdef \set{(i,j) \in R}{(j,i) \in R}$, together with a poset structure~${\prec_R} \eqdef R /{\equiv_R}$ on the equivalence classes of~$\equiv_R$.
Consequently, there is a one-to-one correspondence between preposets on~$[n+1]$ and transitive-free acyclic oriented graphs on subsets of~$[n+1]$ whose vertex set partitions~$[n+1]$: a preposet~$R$ corresponds to the Hasse diagram of the poset~$\prec_R$ on the equivalence classes of~$\equiv_R$, and conversely, an acyclic oriented graph whose vertex set partitions~$[n+1]$ corresponds to its transitive closure.
We define the \defn{normal cone of a preposet}~$R$ on~$[n+1]$ as the polyhedral cone
\[
C\polar(R) \eqdef \bigcap_{(i,j) \in R} \set{\b{u} \in H}{u_i \le u_j}.
\]

The preposets and their normal cones are convenient to describe the normal fan of the permutahedron~$\Perm$. The $k$-dimensional faces of~$\Perm$ correspond to the following equivalent combinatorial objects:
\begin{itemize}
\item the \defn{surjections} from~$[n+1]$ to~$[n+1-k]$,
\item the \defn{ordered partitions} of~$[n+1]$ with $n+1-k$ parts,
\item the \defn{linear preposets} on~$[n+1]$ of rank~$n+1-k$, \ie the preposets~$L$ on~$[n+1]$ whose associated poset~$\prec_L$ is a linear order on the $n+1-k$ equivalence classes of~$\equiv_L$.
\end{itemize}
We pass from surjections to ordered partitions by inversion: the fibers of a surjection from~$[n+1]$ to~$[n+1-k]$ define an ordered partition of~$[n+1]$ with $n+1-k$ parts, and reciprocally the positions of the elements of~$[n+1]$ in an ordered partition of~$[n+1]$ with $n+1-k$ parts define a surjection from~$[n+1]$ to~$[n+1-k]$. For example, the surjection $[2,2,1,1,2,1,2,1]$ (given in one-line notation) corresponds to the ordered partition~$(\{3,4,6,8\}, \{1,2,5,7\})$. In turn, ordered partitions and linear preposets are clearly equivalent.

To illustrate these descriptions of the faces of the permutahedron, observe that the vertices of~$\Perm$ correspond to the $(n+1)!$ permutations of~$[n+1]$, while the facets of~$\Perm$ correspond to the $2^{n+1}-2$ proper and non-empty subsets of~$[n+1]$. We denote by~$\face(L)$ the face of~$\Perm$ that corresponds to a linear preposet~$L$. The normal cone of~$\face(L)$ is precisely the normal cone~$C\polar(L)$.

The following two statements generalize Propositions~\ref{prop:maximalNormalCones} and \ref{prop:maximalCambrianCones} from vertices to all faces of~$\Asso$. In particular, they provide a complete description of the normal fan of~$\Asso$ and of its relation to the braid arrangement. Let~$C\polar(W)$ denote the normal cone of the face~$\face(W)$ of~$\Asso$ that corresponds to a dissection~$W$ of~$\ver{P}$ as described in Lemma~\ref{lem:faceDissection}.

\begin{proposition}[\protect{\cite[Prop.~3.5]{PostnikovReinerWilliams}}]
\label{prop:normalFan}
For any dissection~$W$ of~$\ver{P}$, the normal cone~$C\polar(W)$ is the normal cone of the preposet of the spine~$\spine{W}$.
\end{proposition}

An \defn{extension} of a preposet~$R$ on~$[n+1]$ is a preposet~$\widetilde R$ on~$[n+1]$ containing~$R$, that is, all relations of~$R$ are also relations of~$\widetilde R$. A \defn{linear extension} of~$R$ is an extension~$L$ of~$R$ which is a linear preposet, that is, its associated poset~$\prec_L$ is a linear order on the equivalence classes of~$\equiv_L$.

\begin{proposition}[\protect{\cite[Prop.~3.5]{PostnikovReinerWilliams}}]
\label{prop:CambrianFan}
For any dissection~$W$ of~$\ver{P}$, the normal cone~$C\polar(W)$ in~$\Asso$ is the union of the normal cones~$C\polar(L)$ in~$\Perm$ of the linear extensions~$L$ of the transitive closure~$\prec_W$ of~$\spine{W}$.
\end{proposition}

In other words, the normal cone of any face of~$\Asso$ is a union of normal cones of faces of~$\Perm$. We say that the normal fan of the permutahedron~$\Perm$ \defn{refines} the normal fan of the associahedron~$\Asso$.

\begin{exm}{1}{continued}
For the dissection~$\widetilde W\ex$ of \fref{fig:dissections}\,(right), the normal cone~$C\polar(\widetilde W\ex)$ is the union of the braid cones of the linear extensions of~$\widetilde W\ex$ given by
\begin{gather*}
(\{1,2,3,4,5,6,7,8\}),
\qquad
(\{3,4,5,6,7,8\}, \{1,2\}),
\qquad
(\{1,2,3,4,6,8\}, \{5,7\}), \\
(\{3,4,6,8\}, \{1,2,5,7\}),
\qquad
(\{3,4,6,8\}, \{1,2\}, \{5,7\}),
\qquad
(\{3,4,6,8\}, \{5,7\}, \{1,2\}).
\end{gather*}
The first cone is a point, the next three cones are rays, and the last two cones are $2$-dimensional.
\end{exm}

Any $k$-dimensional face~$F$ of~$\Perm$ corresponds to a surjection~$\tau : [n+1] \to [n+1-k]$. We now present a combinatorial description of the minimal normal cone of~$\Asso$ that contains the normal cone of~$F$. As in the previous Section~\ref{subsec:maximalNormalCones} where we considered the situation~$k=0$, we can use dissections of~$\ver{P}$ as well as spines:
\begin{enumerate}[(i)]
\item Starting from the surjection~$\tau$, we construct a dissection~$\surjectionPermAsso(\tau)$ of~$\ver{P}$ into~$n+1-k$ cells. For~$i \in \{0, \dots, n+1-k\}$, define $\pi_i(\tau)$ to be the $x$-monotone path in~$\ver{P}$ joining the vertices~$\ver{0}$ and~$\ver{n+2}$ and passing through the vertices of the symmetric difference~$D \symdif \tau^{-1}([j])$. The dissection~$\surjectionPermAsso(\tau)$ is the union of the paths~$\pi_0(\tau), \dots, \pi_{n+1-k}(\tau)$. See \fref{fig:pathsDissection}.
\item Starting from the surjection~$\tau : [n+1] \to [n+1-k]$, we construct a leveled spine~$\leveledBijection{\tau}$. We write the surjection~$\tau$ in one-line notation and proceed level by level from bottom to top almost as in the previous section. The only difference is that all positions with the same image~$i$ are treated simultaneously at level~$i$. See \fref{fig:leveledSpinesDissection}.
\end{enumerate}

\begin{exm}{1}{continued}
\label{exm:small7}
Consider the surjection~$\tau\ex \eqdef [2,2,1,1,2,1,2,1]$ (given in one-line notation). \fref{fig:pathsDissection} illustrates the successive steps to construct the dissection~$\surjectionPermAsso(\rho\ex)$, while \fref{fig:leveledSpinesDissection} shows the successive steps to obtain the leveled spine~$\leveledBijection{\tau\ex}$. The resulting dissection and spine are~~$\widetilde W\ex$ and~$\spine{\widetilde {W\ex}}$ of \fref{fig:dissections}\,(right).
\end{exm}

\begin{figure}
  \centerline{\includegraphics[scale=.6]{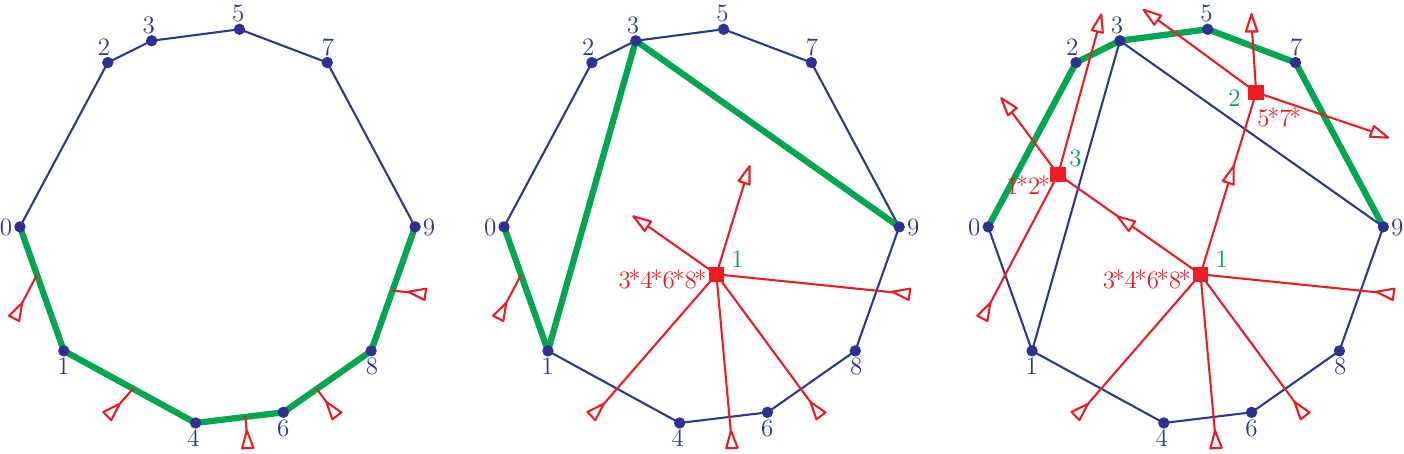}}
  \caption{The sequence of paths~$\pi_0(\tau\ex), \pi_1(\tau\ex), \pi_2(\tau\ex)$ for~$\tau\ex = [2,2,1,1,2,1,2,1]$.}
  \label{fig:pathsDissection}
\end{figure}
\begin{figure}
  \centerline{\includegraphics[scale=1]{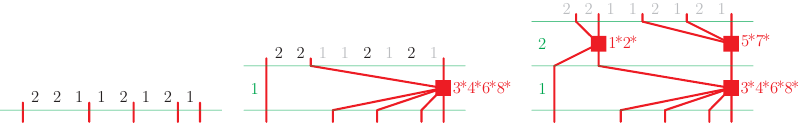}}
  \caption{The construction of the leveled spine~$\leveledBijection{\tau\ex}$ for~$\tau\ex = [2,2,1,1,2,1,2,1]$.}
  \label{fig:leveledSpinesDissection}
\end{figure}

The following statement extends Proposition~\ref{prop:singletons}.

\begin{proposition}
\label{prop:commonFaces}
For a dissection~$W$ of~$\ver{P}$ and a linear preposet~$L$ on~$[n+1]$, the following assertions are equivalent:
\begin{enumerate}[(i)]
\item The spine~$\spine{W}$ is a directed path (with blossoms) labeled by~$L$.
\item The transitive closure of the spine~$\spine{W}$ is~$L$.
\item The affine span of the face~$\face(W)$ of~$\Asso$ coincides with that of the face~$\face(L)$ of~$\Perm$.
\item The normal cone~$C\polar(W)$ of~$\Asso$ coincides with the normal cone~$C\polar(L)$ of~$\Perm$.
\end{enumerate}
\end{proposition}

According to Lemma~\ref{lem:refinement}, these properties are stable by coarsenings and thus satisfied by all faces containing a common vertex of the permutahedron~$\Perm$ and of the associahedron~$\Asso$.

\medskip
To conclude, we extend Proposition~\ref{prop:parallelepiped}. Note that the faces of the parallelepiped~$\Para$ are in correspondence with ternary words of length~$n$ on the alphabet~$\{-1, 0, 1\}$. For any word $w = w_1 \cdots w_n \in \{-1, 0, 1\}^n$, the normal cone of the corresponding face of~$\Para$ is the polyhedral cone~$C\polar(w) \eqdef \set{\b{x} \in H}{w_i(x_i - x_{i+1}) \le 0 \text{ for } i \in [n]}$. The following statement characterizes which normal cones of~$\Asso$ are contained in~$C\polar(w)$.

\begin{proposition}
For any dissection~$W$ of~$\ver{P}$, the labels~$i$ and~$i+1$ are always comparable in~$\spine{W}$. For any ternary word~$w \in \{-1, 0, 1\}^n$, the normal cone~$C\polar(w)$ in~$\Para$ is the union of the normal cones~$C\polar(W)$ of the dissection~$W$ of~$\ver{P}$ with the following property: for all~$i \in [n]$, the label~$i$ appears below~$i+1$ in~$\spine{W}$ if~$w_i = 1$, in the same node if~$w_i = 0$, and above if~$w_i = -1$.
\end{proposition}


\subsection{Weak order, slope increasing flip lattice, and boolean lattice}
\label{sec:latticeQuotient}

To finish this section, we briefly mention the relationship between the weak order on permutations of~$[n+1]$, the slope increasing flip lattice on triangulations of~$\ver{P}$ presented in Section~\ref{subsec:slopeIncreasingFlip}, and the boolean lattice. They can all be seen in terms of orientations of the $1$-skeleta of~$\Perm$, $\Asso$, and~$\Para$, and they are related by the surjections~$\surjectionPermAsso$ and~$\surjectionAssoPara$ studied in this section.

An \defn{inversion} of a permutation~$\sigma$ of~$[n+1]$ is a pair~$i<j$ such that~$\sigma(i) > \sigma(j)$. The \defn{right weak order} on the permutations of~$[n+1]$ is defined as the inclusion order on the inversion sets of the permutations, and induces a lattice structure. The minimal element of this lattice is the identity permutation~$e \eqdef [1, 2, \dots, n+1]$, while the maximal element is~$w_\circ \eqdef [n+1, n, \dots, 2, 1]$. The vector~$\ver{U} \eqdef (n, n-2, \dots, -n+2, -n)$ defined in Section~\ref{subsec:slopeIncreasingFlip} joins the vertices corresponding to these two extremal permutations. Orienting the $1$-skeleton of the permutahedron~$\Perm$ according to the direction~$\ver{U}$ yields the Hasse diagram of the right weak order. Similarly, we have seen that orienting the $1$-skeleton of the associahedron~$\Asso$ according to the direction~$\ver{U}$ yields the increasing flip graph~$\flipGraph(\ver{P})$ on the triangulations of~$\ver{P}$. Finally, orienting the $1$-skeleton of the parallelepiped~$\Para$ according to the direction~$\ver{U}$ yields the Hasse diagram of the boolean lattice. In fact, as observed in~\cite{Reading-CambrianLattices, Reading-latticeCongruences}, these lattices are related by the surjections~$\surjectionPermAsso$ and~$\surjectionAssoPara$.

\begin{proposition}[\protect{\cite[Prop.~6.6]{Reading-CambrianLattices}}]
The surjection~$\surjectionPermAsso$ from the permutations of~$[n+1]$ to the triangulations of~$\ver{P}$ is a lattice homomorphism from the weak order to the slope increasing flip lattice. Similarly, the surjection~$\surjectionAssoPara$ from the triangulations of~$\ver{P}$ to the binary word with~$n$ letters is a lattice homomorphism from the slope increasing flip lattice to the boolean lattice.
\end{proposition}

\begin{exm}{2}{Loday's associahedron, continued}
\label{exm:Loday9}
For Loday's polygon~$\ver{P}\loday$, the maps~$\surjectionPermAsso$ and~$\surjectionAssoPara$ send the weak order to the Tamari lattice to the boolean lattice. These maps were used by J.-L.~Loday and M.~Ronco to construct and study Hopf algebra structures on binary trees~\cite{LodayRonco}.
\end{exm}

These surjections identify the identity permutation~$e$, the minimal triangulation~$\minTriang(\ver{P})$, and the word~$[1]^n$, and similarly the longest permutation~$w_\circ$, the maximal triangulation~$\maxTriang(\ver{P})$, and the word~$[-1]^n$. Moreover, the minimal and maximal permutations in the fiber~$\surjectionPermAsso^{-1}(T)$ of a triangulation~$T$ of~$\ver{P}$ are easily computed: they correspond to the minimal and maximal linear extensions of the spine~$\spine{T}$. For example, the minimal and maximal permutations of the fiber of the triangulation~$T\ex$ of \fref{fig:flip}\,(left) are~$[3,4,2,1,7,6,8,5]$ and~$[7,8,3,2,5,4,6,1]$, respectively.


\section{Spines and Minkowski decompositions}
\label{sec:Minkowski}

Since the normal fan of~$\Asso$ coarsens the normal fan of the permutahedron~$\Perm$, the associahedra studied in this paper fit into the class of \defn{generalized permutahedra} introduced by A.~Postnikov in~\cite{Postnikov}.  Any generalized permutahedron is obtained by gliding facets from the permutahedron~$\Perm$ while staying in its deformation cone and can therefore be described as
\[
\Defo \eqdef \biggset{\b{x} \in \R^{n+1}}{\sum_{j \in [n+1]} x_j = z_{[n+1]} \text{ and } \sum_{j \in J} x_j \ge z_J \text{ for } \varnothing \ne J \subsetneq [n+1]},
\]
for a family~$\{z_J\}_{\varnothing \ne J \subseteq [n+1]} \in \R^{2^{[n+1]}-1}$ such that $z_{[n+1]} = \binom{n+2}{2}$ and $z_I + z_J \le z_{I \cup J} + z_{I \cap J}$ for all non empty $I, J \subseteq [n+1]$. In~\cite{ArdilaBenedettiDoker}, F.~Ardila, C.~Benedetti, and J.~Doker proved that any generalized permutahedron can be decomposed into a Minkowski sum and difference of dilated faces of the standard simplex 
\[
\Defo = \sum_{\varnothing \ne I \subseteq [n+1]} y_I \simplex_I,
\]
where~$\simplex_I \eqdef \conv\set{e_i}{i \in I}$. Here, the Minkowski difference~$P-Q$ of two polytopes~${P,Q \subset \R^{n+1}}$ is defined to be the unique polytope~$R \subset \R^{d+1}$ such that~$P = Q + R$, if it exists. In other words,~${P-Q}$ is only defined if~$Q$ is a Minkowski summand of~$P$. If we assume that all values~$\{z_J\}$ are tight, \ie that they define face supporting hyperplanes of~$\Defo[\{z_J\}]$, then the dilation factors~$\{y_I\}$ relate to the right hand sides~$\{z_J\}$ by M\"obius inversion
\[
z_J = \sum_{I \subseteq J} y_I
\qquad\text{and}\qquad
y_I = \sum_{\varnothing \ne J \subseteq I} (-1)^{|I \ssm J|} z_J.
\]

We now focus on the associahedra~$\Asso$. First, we determine the tight right hand sides~$\{z_J\}$ for~$\Asso$; the proof is a straightforward computation.

\begin{proposition}[\protect{\cite[Prop.~3.8]{Lange}}]
Fix a subset~$J \subseteq [n+1]$, and consider the dissection~$W_J$ of~$\ver{P}$ whose associated face~$\face(W_J)$ maximizes the linear functional~$\b{x} \mapsto \sum_{j \in J} x_j$ on~$\Asso$. Then
\[
z_J = \sum_{\delta \in W_J} \binom{\bel(\delta) + 1}{2} - \binom{n + 2}{2} |J \cap U|,
\]
where the sum runs over all diagonals~$\delta$ of~$W_J$, including boundary ones.
\end{proposition}

Note that the dissection~$W_J$ from this proposition can be computed with the methods presented in Section~\ref{subsec:normalFans} since its spine is~$\spine{W_J} = \surjectionPermAsso(J, [n+1] \ssm J)$.

We could now compute the dilation factors~$\{y_I\}$ in the Minkowski decomposition of~$\Asso$ from the right hand sides~$\{z_J\}$ by M\"obius inversion. As observed by C.~Lange in~\cite{Lange}, it turns out however that the inversion simplifies to closed formulas presented below. In this section, we give a short and self-contained proof for these formulas. Although our proof is largely inspired from spines, we think that the dual presentation matches with the geometric intuition of the reader better. On the other hand, the understanding of this proof in terms of spines is the key to the generalization of Theorem~\ref{theo:MinkowskiDilationFactors} to signed tree associahedra~\cite{Pilaud}. We therefore introduce the following definitions both in the primal and dual picture and use dissections in the proof of Theorem~\ref{theo:MinkowskiDilationFactors}.

We say that a $2$-dimensional polygonal cell~$\ver{C}$ of a dissection~$W$ of~$\ver{P}$ is a \defn{big top} if its lower convex hull is contained in the lower hull of~$\ver{P}$. We think of~$\ver{C}$ as schematic view of a cross-section of a circus tent where the path on the upper convex hull (\ie connecting the left-most with the right-most vertex of~$\ver{C}$) corresponds to the tarpaulin. In terms of spines, the node~$\nod{C}$ of~$\spine{W}$ has no incoming arcs, only incoming blossoms. We denote by~$\delta_\ell(\ver{C})$ (resp.~by~$\delta_r(\ver{C})$) the edge connecting the two leftmost (resp.~rightmost) vertices of~$\ver{C}$, and define the \defn{weight} of~$\ver{C}$ to be
\[
\Omega(\ver{C}) \eqdef (-1)^{|\ver{C} \cap U|}\ell(\ver{C})r(\ver{C}),
\]
where
\[
\ell(\ver{C}) = 
\begin{cases}
-1 & \text{if } \delta_\ell(\ver{C}) \text{ is on the lower hull of } \ver{C} \\
\abo(\delta_\ell(\ver{C})) & \text{otherwise,}
\end{cases}
\]
and similarly for~$r(\ver{C})$. In terms of spines, the value of~$\ell(\ver{C})$ is $-1$ if $\delta_\ell(\ver{C})^*$ is incoming at~$\nod{C}$ and counts the number of blossoms of the subspine determined by the sink of~$\delta_\ell(\ver{C})^*$ if it is outgoing at~$\nod{C}$. By extension, we say that~$I \subseteq [n+1]$ is a big top if it is the set of intermediate vertices of a big top~$\ver{C}_I$ (that is, the index set of~$\nod{C}_I$), and we then define the weight~$\Omega(I) \eqdef \Omega(\ver{C}_I)$. We are now ready to describe the Minkowski decomposition of the associahedron~$\Asso$.

\begin{theorem}[\protect{\cite[Thm.~4.3]{Lange}}]
\label{theo:MinkowskiDilationFactors}
For any~$I \subseteq [n+1]$, the dilation factor~$y_I$ in the Minkowski decomposition of~$\Asso$ is given by
\[
y_I = 
\begin{cases}
0 & \text{if } I \text{ is not a big top,} \\
\Omega(I) + n + 2 & \text{if } I = \{u\} \text{ and } u \in U, \\
\Omega(I) & \text{otherwise.}
\end{cases}
\]
\end{theorem}

\begin{proof}
It is sufficient to check that~$\sum_{I \subseteq J} y_I = z_J$ for any~$J \subseteq [n+1]$. Fix a set~$J \subseteq [n+1]$ and consider the dissection~$W_J$ whose associated face~$\face(W_J)$ maximizes the functional~$\b{x} \mapsto \sum_{j \in J} x_j$. For any diagonal~$\delta = \ver{i}\ver{j}$ of~$W_J$ with~$i < j$, we define
\[
\lef_\delta \eqdef |\Bel(\delta) \cap [i]|, \qquad \mi_\delta \eqdef |\Bel(\delta) \cap [i+1, j-1]| \qquad \text{and}\qquad\rig_\delta \eqdef |\Bel(\delta) \ssm [j-1]|.
\]
We observe that:
\begin{enumerate}[(i)]
\item We have~$\bel(\delta) = |\Bel(\delta)| = \lef_\delta + \mi_\delta + \rig_\delta$.
\item For any big top~$I \subseteq J$, the cell~$\ver{C}_I$ is contained in a single cell of~$W_J$.
\item If~$J$ has an up vertex strictly in between two diagonals~$\delta_\ell$ and~$\delta_r$ of~$\ver{P}$, then the sum of the weights of all big tops~$I \subseteq J$ of~$\ver{P}$ with prescribed~$\delta_\ell(\ver{C}_I) = \delta_\ell$ and~$\delta_r(\ver{C}_I) = \delta_r$ vanishes. Indeed, if~$u$ is such a vertex, we can partition the big tops into pairs~$\big\{I, I \symdif \{u\} \big\}$ whose weights cancel since~$\Omega(I) = - \Omega(I \symdif \{u\})$.
\end{enumerate}
We now want to compute~$\sum_{I \subseteq J} \Omega(I)$. We use Observations~(ii) and~(iii) above to restrict the sum to big tops whose intermediate vertices are all located between or on the vertical lines through the endpoints of a diagonal of~$W_J$. For the diagonal~$\delta = \ver{i}\ver{j}$, we partition the big tops~$I \subseteq J \cap [i,j]$ into those which
\begin{itemize}
\item contain neither~$i$ nor~$j$: there are~$\binom{\mi_\delta + 1}{2}$ such big tops, all of weight~$1$;
\item contain $i$ but not~$j$ and at least one other point: there are~$\mi_\delta$ such big tops, all of weight~$\lef_\delta$, if~$i \in U$ and none otherwise;
\item contain~$j$ but not~$i$ and at least one other point: there are~$\mi_\delta$ such big tops, all of weight~$\rig_\delta$, if~$j \in U$ and none otherwise;
\item contain both~$i$ and~$j$: there is one such big top of weight~$\lef_\delta\rig_\delta$ if~$i,j \in U$ and none otherwise;
\item equal~$\{i\}$: when~$i \in D$ there is no such big top; when~$i \in U$ there is only one of weight~$\Omega(\{i\}) = -\lef_\delta(n+2-\lef_\delta)$, but it is counted in both diagonals of~$W_J$ incident~to~$i$;
\item equal~$\{j\}$: when~$j \in D$ there is no such big top; when~$j \in U$ there is only one of weight~$\Omega(\{j\}) = -(n+2-\rig_\delta)\rig_\delta$, but it is counted in both diagonals of~$W_J$ incident~to~$j$.
\end{itemize}
Summing all contributions, it follows that
\begin{align*}
\sum_{I \subseteq J} \Omega(I)
& = \sum_{\delta \in W_J} \bigg( \!\! \binom{\mi_\delta + 1}{2} + \lef_\delta \mi_\delta + \mi_\delta \rig_\delta + \lef_\delta \rig_\delta - \frac{1}{2} \lef_\delta (n + 2 - \lef_\delta) - \frac{1}{2} (n + 2 - \rig_\delta) \rig_\delta \!\! \bigg) \\
& = \sum_{\delta \in W_J} \bigg( \!\! \binom{\lef_\delta + \mi_\delta + \rig_\delta + 1}{2} - \frac{n+3}{2} ( \lef_\delta + \rig_\delta ) \!\! \bigg) \\
& = \sum_{\delta \in W_J} \binom{\bel(\delta) + 1}{2} - \frac{(n+3)(n+2)|J \cap U|}{2}.
\end{align*}
For the last equality, we used Observation~(i) above and the fact that~$\lef_{\delta'} + \rig_{\delta''} = n+2$ for any two diagonals~$\delta'$ and~$\delta''$ sharing an endpoint in~$J \cap U$. We therefore obtain
\[
\sum_{I \subseteq J} y_I = \sum_{I \subseteq J} \Omega(I) + (n+2)|J \cap U| = \sum_{\delta \in W_J} \binom{\bel(\delta) + 1}{2} - \binom{n+2}{2}|J \cap U| = z_J.
\qedhere
\]
\end{proof}

\begin{exm}{2}{Loday's associahedron, continued}
\label{exm:Loday10}
For Loday's polygon~$\ver{P}\loday$, the dilation factors are given by~$y_I = 1$ if~$I$ is an interval of~$[n+1]$, and~$y_I = 0$ otherwise~\cite{Postnikov}.
\end{exm}


\section{Spines and invariants}
\label{sec:invariants}

This section is devoted to two further properties of~$\Asso$ which are invariant if the combinatorics of the $(n+3)$-gon~$\ver{P}$ changes. First, we show that the vertex barycenters of all associahedra~$\Asso$ coincide with that of the permutahedron~$\Perm$. This surprising invariant was observed in~\cite{Loday} for J.-L.~Loday's associahedron, generalized in~\cite{HohlwegLange} and proved in~\cite{HohlwegLortieRaymond}. We use spines to achieve a shorter proof of this result. Second, we relate spines to the Narayana numbers, using the fact that the $f$- and $h$-vectors of~$\Asso$ are independent of the choice of the $(n+3)$-gon~$\ver{P}$.


\subsection{Barycenters}
\label{subsec:barycenter}

In this section, we give an alternative proof of the following statement that was first proved by C.~Hohlweg, J.~Lortie and A.~Raymond~\cite{HohlwegLortieRaymond}, see Section~\ref{sec:existingWork} for a discussion on previous approaches and for references.

\begin{theorem}[\protect{\cite[Thm.~3.3]{HohlwegLortieRaymond}}]
\label{theo:barycenter}
The vertex barycenters of the associahedron~$\Asso$ and of the permutahedron~$\Perm$ coincide.
\end{theorem}

Recall that the barycenter of the permutahedron~$\Perm$ is the point~$\big(\frac{n+2}{2}, \dots, \frac{n+2}{2} \big)$. The dihedral group~$\dihedral$ of order~$2(n+3)$ acts combinatorially on the diagonals of the \mbox{$(n+3)$-gon}~$\ver{P}$, and therefore on its triangulations. The following statement implies Theorem~\ref{theo:barycenter} and was already observed and proven by~\cite{HohlwegLortieRaymond}. We provide a new elementary proof based on spines.

\begin{theorem}[\protect{\cite[Thm.~3.4]{HohlwegLortieRaymond}}]
\label{theo:barycenterOrbits}
For any triangulation~$T$ of~$\ver{P}$, the barycenter of the vertices of~$\Asso$ corresponding to the triangulations of the orbit~$\dihedral T$ coincides with the barycenter of the permutahedron~$\Perm$.
\end{theorem}

\begin{proof}
Since we work with all the triangulations of the orbit~$\dihedral T$ at the same time, we refine the earlier notation~$\chor(j)$ and write~$\chor_T(j)$ for the number of maximal paths in the spine~$\spine{T}$ whose edge orientation is reversed at node~$\nod{j}$.

According to the definition of the vertices of~$\Asso$ and as observed in~\cite{HohlwegLortieRaymond}, it suffices to prove that
\[
\sum_{\mu \in \dihedral} |\chor_{\mu(T)}(j)| = (n+2)(n+3),
\]
for any given~$j \in [n+1]$. On the right-hand-side of the previous expression, we recognize the number of oriented maximal paths in the spine~$\spine{T}$. Consider such a path~$\pi$ from a blossom~$u$ to a blossom~$v$. Let~$\delta_1, \dots, \delta_\ell$ denote the diagonals of~$T$ such that~$\pi = \delta_1^*\dots\delta_\ell^*$. Here, both~$\delta_1$ and~$\delta_\ell$ are boundary diagonals of~$\ver{P}$. For~$k \in [\ell]$, we let~$n_k$ be the number of vertices of~$\ver{P}$ in the open half-space delimited by~$\delta_k$ and containing the initial blossom~$u$. Observe that~$n_k$ is strictly increasing from~$n_1 = 0$ to~$n_\ell = n+1$. We also denote by~$i_k$ the common endpoint of the diagonals~$\delta_k$ and~$\delta_{k+1}$.

\begin{figure}
  \vspace*{.2cm}
  \centerline{
	\begin{overpic}[width=\textwidth]{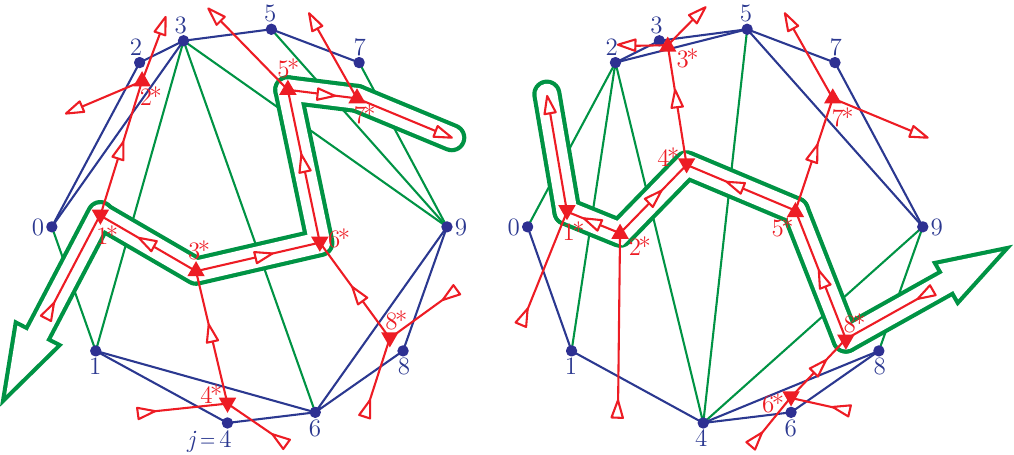}
	    \put(42.5,26){\color{darkgreenpic}$\delta_1$}
	    \put(37.5,29){\color{darkgreenpic}$\delta_2$}
	    \put(36,25){\color{darkgreenpic}$\delta_3$}
	    \put(35,22.5){\color{darkgreenpic}$n_3 \! = \! 2$}
	    \put(17.2,45.7){\color{darkbluepic}$i_3$}
	    \put(17.55,43.4){\color{darkbluepic}\rotatebox{90}{$=$}}
	    \put(28.4,13.5){\color{darkgreenpic}$\delta_4$}
	    \put(29.2,11){\color{darkgreenpic}$n_4 \! = \! 4$}
	    \put(11.3,14){\color{darkgreenpic}$\delta_5$}
	    \put(6.5,11.5){\color{darkgreenpic}$\delta_6$}
  	\end{overpic}
  }
  \caption{A maximal path~$\pi$ in the spine of the triangulation~$T\ex$ of \fref{fig:flip}\,(left). Considering the down label~$j = 4$, we have~$k = 3$ and~$i_k = 3$ is to the right of~$\pi$. Therefore, the combinatorial rotation of the triangulation~$T\ex$ which sends vertex~$3$ to vertex~$4$ sends~$\pi$ to a path with two incoming arcs at~$\nod{4}$.}
  \label{fig:barycenter}
\end{figure}

Let~$k \in [\ell-1]$ be such that~$n_k < j \le n_{k+1}$. Assume first that~$j$ is a down label and that~$i_k$ is to the right of the oriented path~$\pi$. This situation is illustrated on \fref{fig:barycenter}. Let~$\rho$ denote the rotation of~$\dihedral$ which sends index~$i_k$ to~$j$. We claim that the path~$\rho(\pi)$ in the spine~$\spine{\rho(T)}$ changes orientation at node~$\nod{j}$ and traverses from left to right at that node. Indeed, since~$n_k < j$, the diagonal~$\rho(\delta_k)$ is incident to~$j$ on its left, and since~$j \le n_{k+1}$, the diagonal~$\rho(\delta_{k+1})$ is incident to~$j$ on its right. This implies that~$\rho(\delta_k)$ and~$\rho(\delta_{k+1})$ are the left and right bottom diagonals of the triangle~$\nod{j}$ of~$\rho (T)$, and thus that~$\rho(\delta_k)^*$ and~$\rho(\delta_{k+1})^*$ are respectively the left and the right incoming arcs of~$\nod{j}$ in~$\spine{\rho(T)}$. Observe moreover that~$\rho$ is the unique element of~$\dihedral$ such that the path~$\rho(\pi)$ in the spine~$\spine{\rho(T)}$ changes orientation while traversing from left to right at node~$\nod{j}$. With the same arguments, we prove that~$\rho$ is the unique element of~$\dihedral$ such that the path~$\rho(\pi)$ in the spine~$\spine{\rho(T)}$ changes orientation while traversing from left to right at node~$\nod{j}$ in the situation when~$j$ is an up label and~$i_k$ is on the left of the oriented path~$\pi$. Finally, if either $j$ is an up label or~$i_k$ lies on the left of the oriented path~$\pi$, but not both, we prove similarly that the reflection~$\tau$ of~$\dihedral$ which sends~$i_k$ to~$j$ is the unique transformation of~$\dihedral$ such that the path~$\tau(\pi)$ in the spine~$\spine{\tau(T)}$ changes orientation while traversing from left to right at node~$\nod{j}$.

Therefore, for each oriented maximal path~$\pi$ in~$\spine{T}$, we obtain a unique element~$\mu_\pi$ of~$\dihedral$ such that~$\mu_\pi(\pi)$ changes orientation while traversing from left to right at node~$\nod{j}$. Moreover, for any~$\pi \ne \pi'$, we clearly have either~$\mu_\pi \ne \mu_{\pi'}$ or~$\mu_\pi(\pi) \ne \mu_{\pi'}(\pi')$. This ensures that, for all~$j \in [n+1]$,
\[
\sum_{\mu \in \dihedral} |\chor_{\mu(T)}(j)| = (n+2)(n+3). \qedhere
\]
\end{proof}


\subsection{Narayana numbers}

In this section, we use that the entry~$h_\ell(\Asso)$ of the $h$-vector of the associahedron yields the Narayana number~$\Narayana(n,\ell)$, see~\cite[Sect.~10.2]{PostnikovReinerWilliams} and~\cite{Narayana}. As a consequence, we relate combinatorial properties of spines to the Narayana numbers. 

The \defn{$f$-vector} of a polytope~$P \subset \R^n$ is the vector~$f(P) \eqdef (f_0(P), f_1(P), \dots, f_n(P))$ whose $k$\ordinal{} coordinate~$f_k(P)$ is the number of $k$-dimensional faces of~$P$. For example, $f_k(\Perm)$ is the number of ordered partitions of~$[n+1]$ with~$n+1-k$ parts, while~$f_k(\Asso)$ is the number of dissections of~$\ver{P}$ into~$n+1-k$ cells. The \defn{$f$-polynomial} of~$P$ is the polynomial~$f_P(X) \eqdef \sum_{k = 0}^n f_k(P) \, X^k$.

The \defn{$h$-vector}~$h(P) \eqdef (h_0(P), h_1(P), \dots, h_n(P))$ and the \defn{$h$-polynomial}~$h_P(X) \eqdef \sum_{\ell = 0}^n h_\ell(P) \, X^\ell$ of a \emph{simple} polytope~$P$ are defined by the relation
\[
f_P(X) = h_P(X+1),
\]
or equivalently by the equalities
\[
f_k(P) = \sum_{\ell=0}^n \binom{\ell}{k} h_\ell(P) \quad \text{for } 0 \le k \le n
\qquad\text{or}\qquad
h_\ell(P) = \sum_{k=0}^n (-1)^{k+\ell} \binom{k}{\ell} f_k(P) \quad \text{for } 0 \le \ell \le n.
\]
For example, $h_\ell(\Perm)$ is the \defn{Eulerian number}
\[
h_\ell(\Perm)=\Euler(n+1,\ell)=\sum_{k=0}^\ell(-1)^k{{n+2}\choose{k}}(\ell+1-k)^{n+1},
\]
that is, the number of permutations of~$[n+1]$ with $\ell$ descents, while~$h_\ell(\Asso)$ is the \defn{Narayana number}~\cite{Narayana}
\[
h_\ell(\Asso) = \Narayana(n,\ell) \eqdef \frac{1}{n} \binom{n}{\ell} \binom{n}{\ell-1}.
\]

Given any simple polytope~$P \subset \R^n$ and any generic linear functional~$\psi \in (\R^n)^*$, the $\ell$\ordinal{} entry~$h_\ell(P)$ of the $h$-vector of~$P$ equals the number of vertices of out-degree~$\ell$ in the $1$-skeleton of~$P$ oriented by increasing values of~$\psi$. In particular, this ensures that the $h$-vector is symmetric (since it gives the same vector for the functionals~$\psi$ and~$-\psi$). Application of the linear functional~$\b{x} \mapsto \dotprod{\b{U}}{\b{x}}$ to the $1$-skeleton of~$\Asso$ yields the following description of the Narayana numbers using Lemma~\ref{lem:slopeIncreasing}. We say that an arc of a spine~$\spine{T}$ is \defn{ordered} if its source is smaller than its target, and \defn{reversed} otherwise.

\begin{proposition}
\label{prop:Narayana}
For any polygon~$\ver{P}$, the number of triangulations~$T$ of~$\ver{P}$ whose spine~$\spine{T}$ has $\ell$ ordered arcs and $n-\ell$ reversed arcs is the Narayana number~$\Narayana(n,\ell)$.
\end{proposition}

The symmetry of the $h$-vector immediately leads to the following corollary.

\begin{corollary}
For any polygon~$\ver{P}$, there are as many triangulations~$T$ of~$\ver{P}$ whose spine~$\spine{T}$ has $\ell$ ordered arcs and $n-\ell$ reversed arcs as triangulations~$T$ of~$\ver{P}$ whose spine~$\spine{T}$ has $n-\ell$ ordered arcs and $\ell$ reversed arcs.
\end{corollary}

\begin{exm}{2}{Loday's associahedron, continued}
\label{exm:Loday11}
For Loday's polygon~$\ver{P}\loday$, we obtain the classical description of the Narayana number~$\Narayana(n,\ell)$ as the number of (rooted) binary trees on~$n+1$ nodes with $\ell$ right children and $n-\ell$ left children.
\end{exm}


\section{Spines and existing work}
\label{sec:existingWork}

To conclude this paper, we relate our results to existing work. This enables us to give some taste of the various connections of the associahedron with different mathematical areas, to properly refer to the literature, and to indicate some directions of further research.

\paragraph{Contact graphs of pseudoline arrangements on sorting networks} 
Motivated by polytopal realizations of generalizations of the associahedron (in relation with multitriangulations of polygons), V.~Pilaud and F.~Santos defined the family of brick polytopes of sorting networks~\cite{PilaudSantos-brickPolytope}. A \defn{sorting network}~$\cN$ is a sequence of simple transpositions~${(i, i+1)}$ of~$\fS_{n+1}$. We are interested in all possible reduced expressions for the longest permutation~$w_\circ \in \fS_{n+1}$ which are subsequences of~$\cN$. A network can be visually represented by a geometric diagram, formed by $n+1$ horizontal lines (levels) together with some vertical segments between two consecutive levels (commutators) corresponding to the simple transpositions of the network. The subsequences of the network~$\cN$ which form reduced expressions for~$w_\circ$ appear as crossing points of pseudoline arrangements on the diagram of~$\cN$. V.~Pilaud and F.~Santos associate to each such pseudoline arrangement~$\Lambda$
\begin{enumerate}[(i)]
\item its brick vector~$\brick(\Lambda) \in \R^{n+1}$, whose $i$\ordinal{} coordinate counts the bricks (cells of the network) below the $i$\ordinal{} pseudoline of~$\Lambda$ --- thus explaining the name --- and
\item its (oriented and labeled) contact graph~$\contact{\Lambda}$, which encodes the contacts between the pseudolines of~$\Lambda$.
\end{enumerate}
The brick polytope~$\Brick(\cN)$ of the sorting network~$\cN$ is the convex hull of all the points~$\brick(\Lambda)$. Its properties are determined by the contact graphs~$\contact{\Lambda}$ of the pseudolines arrangements supported by~$\cN$. In particular, the brick vector~$\brick(\Lambda)$ is a vertex of~$\Brick(\cN)$ if and only if the contact graph~$\contact{\Lambda}$ is acyclic, and the normal cone of~$\brick(\Lambda)$ is then the braid cone of the transitive closure of~$\contact{\Lambda}$.

Based on the duality described in~\cite{PilaudPocchiola}, it was observed in~\cite{PilaudSantos-brickPolytope} that all associahedra of C.~Hohlweg and C.~Lange~\cite{HohlwegLange} appear as brick polytopes of certain well-chosen networks. These networks are obtained from the dual pseudoline arrangement of an $(n+3)$-gon~$\ver{P}$ by removing its external hull. A triangulation~$T$ of~$\ver{P}$ corresponds to a pseudoline arrangement~$T^\star$ supported by the diagram corresponding to~$\ver{P}$, and the contact graph of~$T^\star$ is precisely the spine~$\spine{T}$ of~$T$ where we remove all blossoms. Our definition of spines of triangulations is thus a direct primal description of contact graphs of pseudoline arrangements supported by these specific networks.

\paragraph{Graph associahedra, nestohedra, generalized permutahedra}
The graph associahedron of a graph~$G$ is a polytope whose faces correspond to collections of connected subgraphs of~$G$ which are pairwise nested, or disjoint and non-adjacent. Its first polytopal realization was described by M.~Carr and S.~Devadoss~\cite{CarrDevadoss, Devadoss}, and later with other techniques by A.~Postnikov~\cite{Postnikov} and A.~Zelevinsky~\cite{Zelevinsky}. The associahedra~$\Asso$ are graph associahedra for paths. 

Generalizing the construction of C.~Hohlweg and C.~Lange~\cite{HohlwegLange}, V.~Pilaud defines the signed nested complex of a signed tree~$T$ and constructs a signed tree associahedron realizing this simplicial complex~\cite{Pilaud}. The main tool of this construction is the definition of the spine of a nested collection of signed tubes of~$T$, which is based on the spines introduced and developed in this paper.

Graph associahedra are graphical examples of nestohedra, which were studied in particular in~\cite{Postnikov, Zelevinsky}. In turn, nestohedra appear in the work of A.~Postnikov~\cite{Postnikov} as specific examples of generalized permutahedra, obtained from the permutahedron~$\Perm$ by gliding certain facets. Further properties of generalized permutahedra, in particular their relations with the braid arrangement, their normal fans, and their face vectors, were later studied in detail by A.~Postnikov, V.~Reiner and L.~Williams~\cite{PostnikovReinerWilliams}. Slightly extending the arguments in Postnikov's proof, the authors observed that it sometimes suffices to glide some facets to infinity without changing the other facets. This observation provides one more evidence that the construction of C.~Hohlweg and C.~Lange~\cite{HohlwegLange} might be extended to signed versions of nestohedra. We refer to \cite{Pilaud-removahedra} for further details.

\paragraph{Generalized associahedra and cluster algebras}
Generalized associahedra are polytopal realizations of finite type cluster complexes, which arise from the rich combinatorial theory of cluster algebras initiated by S.~Fomin and A.~Zelevinsky in~\cite{FominZelevinsky-clusterAlgebrasI, FominZelevinsky-clusterAlgebrasII}. Finite type cluster algebras are classified by the Cartan-Killing classification for finite crystallographic root systems. For example, the associahedra~$\Asso$ realize the cluster complexes of type~$A$. 

The first polytopal realizations of generalized associahedra were constructed by F.~Chapoton, S.~Fomin and A.~Zelevinsky in~\cite{ChapotonFominZelevinsky}. Extending the construction of~\cite{HohlwegLange}, C.~Hohlweg, C.~Lange and H.~Thomas~\cite{HohlwegLangeThomas} obtained various polytopal realizations, based on geometric and combinatorial properties of Cambrian fans presented below. These realizations are obtained by gliding to infinity certain facets of the Coxeter permutahedron, which is defined as the convex hull of a generic point under the action of the underlying finite Coxeter group. Generalizing the approach of~\cite{ChapotonFominZelevinsky}, S.~Stella~\cite{Stella} recovered these realizations of~\cite{HohlwegLangeThomas} and made precise connections between the geometric properties of these generalized associahedra and the algebraic and combinatorial properties of the corresponding cluster algebras. In their recent preprint~\cite{IgusaOstroff}, K.~Igusa and J.~Ostroff study ``mixed cobinary trees'', which are precisely the spines of the triangulations of~$\ver{P}$. They use these trees to revisit type~$A$ cluster mutation process and its connection to quiver representation theory.

In a different direction, V.~Pilaud and C.~Stump~\cite{PilaudStump-brickPolytope} recently generalized the brick polytope construction of V.~Pilaud and F.~Santos~\cite{PilaudSantos-brickPolytope} mentioned above to spherical subword complexes on finite Coxeter groups, defined by A.~Knutson and E.~Miller~\cite{KnutsonMiller-subwordComplex}. As in type~$A$, cluster complexes can be described as subword complexes of well-chosen words~\cite{CeballosLabbeStump}, and the brick polytopes of these subword complexes provide new descriptions of the generalized associahedra of~\cite{HohlwegLangeThomas}. The main tool for this construction is the root configuration of a facet of the subword complex. In type~$A$, this root configuration is formed by the column vectors of the incidence matrix of the spine of the corresponding triangulation.

\paragraph{Cambrian lattices and fans}
We have seen that the normal fan of the associahedron~$\Asso$ is refined by that of the permutahedron, thus defining a surjective map~$\surjectionPermAsso$ from the permutahedron of~$[n+1]$ to the triangulations of~$\ver{P}$. Moreover, the quotient of the weak order under this map gives the increasing flip lattice, whose Hasse diagram is isomorphic to the $1$-skeleton of the associahedron~$\Asso$ oriented in the direction~$\ver{U}$. In Section~\ref{sec:normalFan} we described how spines combinatorially encode the geometry of the normal fan and the sujection~$\surjectionPermAsso$.

Based on N.~Reading's work on lattice congruences of the weak order on a finite Coxeter group, N.~Reading and D.~Speyer extended these results to the context of finite type cluster algebras~\cite{Reading-latticeCongruences, Reading-CambrianLattices, ReadingSpeyer}. They constructed the Cambrian fans providing various geometric realizations of any cluster complex of finite type~$W$. Subsequently, C.~Hohlweg, C.~Lange and H.~Thomas constructed generalized associahedra with some Cambrian fan as normal fan~\cite{HohlwegLangeThomas}. These fans are refined by the Coxeter fan of~$W$, thus defining a surjective map from~$W$ to the clusters of the cluster algebra for~$W$. The quotient of the weak order on~$W$ under this map is the Cambrian lattice studied by N.~Reading~\cite{Reading-CambrianLattices}, whose Hasse diagram is isomorphic to the $1$-skeleton of the generalized associahedron oriented in the direction~$\ver{U}$. 

In a recent paper~\cite{PilaudStump-ELlabeling}, V.~Pilaud and~C.~Stump studied natural generalizations of the increasing flip order to arbitrary subword complexes. They described in particular four canonical spanning trees of the flip graph of a subword complex, which led to efficient enumeration algorithms of their facets. These results can be applied to obtain efficient enumeration schemes of the triangulations of a polygon.

\paragraph{Barycenters}
We now briefly discuss the history and the different strategies adopted to prove the surprising property that the associahedron~$\Asso$ and the permutahedron~$\Perm$ have the same vertex barycenters. According to J.-L.~Loday~\cite{Loday}, this property was first observed by F.~Chapoton for Loday's associahedron. By way of proof, it is claimed in that same paper~\cite{Loday} that any triangulation~$T$ of~$\ver{P}\loday$ contributes to the barycenter of~$\Asso[\ver{P}\loday]$ as much as the permutations of its fiber~$\surjectionPermAsso^{-1}(T)$ contribute to the barycenter of~$\Perm$. This property would clearly imply that the barycenters of~$\Asso$ and~$\Perm$ coincide, since the projection map~$\surjectionPermAsso$ is surjective from the permutations of~$[n+1]$ to the triangulations of~$\ver{P}$. The argument is however misleading since it already fails as soon as~$n \ge 3$. J.-L.~Loday, informed by C.~Hohlweg around 2007, encouraged him to find a correct argument.

A complete proof of the equality of the barycenters was given for all associahedra~$\Asso$ by C.~Hohlweg, J.~Lortie and A.~Raymond in~\cite{HohlwegLortieRaymond}. The key observation of their proof is that the barycenter of the points of~$\Asso$ corresponding to the orbit of a given triangulation under the natural action of the dihedral group on~$\ver{P}$ already coincides with the barycenter of~$\Perm$. To prove this stronger statement, their method requires two steps: first, they prove the result for Loday's associahedron using an inductive argument on the number of vertices of the polygon~$\ver{P}\loday$, and second, they use a well-chosen map to send an arbitrary realization~$\Asso$ to Loday's realization~$\Asso[\ver{P}\loday]$ preserving the barycenter.

In~\cite{PilaudStump-barycenter}, V.~Pilaud and C.~Stump provide a different proof of the barycenter property, which extends to all finite Coxeter groups and all fairly balanced permutahedra (meaning not only the permutahedron~$\Perm$, but more general ones). They also show that the barycenter of the vertices of the generalized associahedron corresponding to certain orbits of clusters under a natural action already vanishes. These extentions were conjectured by C.~Hohlweg, C.~Lange and H.~Thomas in~\cite{HohlwegLangeThomas, Hohlweg}. The proof in~\cite{PilaudStump-barycenter}, based on brick polytopes~\cite{PilaudSantos-brickPolytope, PilaudStump-brickPolytope}, also requires two steps: it is first proved that all generalized associahedra have the same vertex barycenter, and then that the barycenter of the superposition of the vertex sets of two well-chosen generalized associahedra is the origin.

To our knowledge, the proof presented in this paper is the first direct and elementary proof of the barycenter property for associahedra~$\Asso$ of type~$A$. Although it highly relies on C.~Hohlweg, J.~Lortie and A.~Raymond's observation on barycenters of dihedral orbits, our interpretation of C.~Hohlweg and C.~Lange's associahedra in terms of spines of the triangulations is the key ingredient to simplify the computations.

\paragraph{Cambrian Hopf algebras}
Hopf algebras are vector spaces endowed with a structure of algebra and a structure of coalgebra which are compatible to each other. Such structures appear in particular in combinatorics. Relevant examples are the Hopf algebra on permutations by C.~Malvenuto and C.~Reutenauer~\cite{MalvenutoReutenauer}, the Hopf algebra on binary trees by J.-L.~Loday and M.~Ronco~\cite{LodayRonco}, and the descent algebra on binary sequences by L.~Solomon~\cite{Solomon}. These algebras are closely related to the geometry of the permutahedron~$\Perm$, of Loday's associahedron~$\Asso[\ver{P}\loday]$, and of the parallelepiped~$\Para$.

Extending these constructions, G.~Chatel and V.~Pilaud defined the Cambrian Hopf algebra~\cite{ChatelPilaud}, with basis indexed by spines of triangulations. The geometric and combinatorial properties of spines developed in the present paper are instrumental for the algebraic construction. For example, the definition of the surjections~$\surjectionPermAsso$ and $\surjectionAssoPara$ are at the heart of the construction of the Cambrian algebra. We refer the reader to~\cite{ChatelPilaud} for further details.

\paragraph{Acyclic sets}
The objects studied in this paper have further connections to more surprising topics, which we illustrate here by a single example. In order to rank a list of candidates, voters of a constituency can practice \defn{pairwise majority voting}: each voter chooses a ranking of all candidates, and each pair of candidates is then ranked according to the majority preference of the voters. To avoid intransitive results, one possible solution is to force voters to make their choice among a restricted set of rankings. A subset of permutations which guarantees intransitivity under pairwise majority voting is called an \defn{acyclic set of linear orders}. A natural but still open question is to identify acyclic sets of linear orders of maximum cardinality for a given number of candidates. In the context of higher Bruhat orders, {\'A}.~Galambos and V.~Reiner~\cite{GalambosReiner} studied acyclic sets from a combinatorial perspective and provided a unified description of certain acyclic sets known to social choice theorists. Surprisingly, these specific acyclic sets have an interpretation in our context: each such set corresponds to the set of common vertices of the permutahedron with some associahedron~$\Asso$ studied in~\cite{HohlwegLange}, or, by Proposition~\ref{prop:singletons}, to spines of triangulations of~$\ver{P}$ that are directed paths. Using this interpretation and the framework of sorting networks described earlier, J.-P.~Labb{\'e} and C.~Lange recently derived explicit formulas for the cardinality of these specific acyclic sets~\cite{LabbeLange}. This generalizes a formula of {\'A}.~Galambos and V.~Reiner~\cite{GalambosReiner} for Fishburn's alternating scheme, which turns out to be the largest acyclic set in this family. These formulas can also be extended to count common vertices of Coxeter permutahedra and associahedra of~\cite{HohlwegLangeThomas} for other finite Coxeter groups.


\section*{Acknowledgement}

We are grateful to Christophe Hohlweg for his deep insight on the content of this paper. While discussing a preliminary draft, he suggested to describe directly the surjection from permutations to spines and worked out the details with us, as presented in Section~\ref{subsec:maximalNormalCones}. Fruitful discussions with him also led to the short proof of the barycenter invariance in Section~\ref{subsec:barycenter}. We also thank Cesar Ceballos and all anonymous referees for helpful comments on earlier versions of this manuscript.


\bibliographystyle{alpha}
\bibliography{HLassociahedraRevisited}
\label{sec:biblio}

\end{document}